\documentclass{article} 
\usepackage{macros}

\title{A Gradient Smoothed Functional Algorithm with Truncated Cauchy Random
Perturbations for Stochastic Optimization}

\author{
Akash Mondal\\
{\normalsize Indian Institute of Science} \\
{\normalsize \texttt{akashmondal@iisc.ac.in}}
\and 
Prashanth L. A.\\
{\normalsize Indian Institute of Technology Madras}\\
{\normalsize \texttt{prashla@cse.iitm.ac.in}}
\and
Shalabh Bhatnagar\\
{\normalsize Indian Institute of Science} \\
{\normalsize \texttt{shalabh@iisc.ac.in}}
}
\date{}
\begin{document}

\maketitle

\begin{abstract}
    In this paper, we present a stochastic gradient algorithm for minimizing a smooth objective function that is an expectation over
noisy cost samples, and only the latter are observed for any given parameter. Our algorithm employs a gradient estimation scheme with random perturbations, which are formed using the truncated Cauchy distribution from the $\delta$ sphere. We analyze the bias and variance of the proposed gradient estimator. Our algorithm is found to be particularly useful in the case when the objective function is non-convex, and the parameter dimension is high. From an asymptotic convergence analysis, we establish that our algorithm converges almost surely to the set of stationary points of the objective function and obtain the asymptotic convergence rate. We also show that our algorithm avoids unstable equilibria, implying convergence to local minima. Further, we perform a non-asymptotic convergence analysis of our algorithm. In particular, we establish here a non-asymptotic bound for finding an $\epsilon$-stationary point of the non-convex objective function. Finally, we demonstrate numerically through simulations that the performance of our algorithm outperforms GSF, SPSA and RDSA by a significant margin over a few non-convex settings and further validate its performance over convex (noisy) objectives.
% \textbf{Keywords:} stochastic approximation, stochastic optimization, non-convex optimization,Cauchy SF, gradient descent, gradient estimation.
\end{abstract}

\section{Introduction}
% \label{submission}
In this paper, we consider the following stochastic optimization (SO) problem:
\begin{equation}
\label{def:eq1}
    \textrm{ Find } f^*:=\inf_{x\in\R^d}\left\{f(x) =\int_\Xi F(x,\xi)dP(\xi)\right\},
\end{equation}
where $f\colon \mathbb{R}^d\to \mathbb{R}$ is a smooth function that could be highly nonlinear with multiple local minima, $\xi$ is the noise random variable (r.v.) with support $\Xi$, and $F(x,\xi)$ is a noisy observation of the function value $f(x)$. We do not assume precise gradient information is available. Instead, gradients need to be estimated using the aforementioned noisy observations at certain parameter values.

Stochastic approximation (SA) is an important technique for solving SO problems. 
\cite{RobbinsMonro51} first proposed the SA approach for the problem of root finding and \cite{kiefer} presented the first application of SA for solving SO problems.
Many popular incremental-update procedures for root finding involving noisy function observations are SA algorithms. These algorithms are employed for back-propagation in neural networks \citep{gawthrop}, for solving least squares objectives \citep{yao}, and finding optimal policies in reinforcement learning problems \citep{bertsekas2019}. As a result, advancements in general SA methodology have an impact on a wide range of applications.

We consider in this paper a stochastic gradient  (SG) algorithm based on gradient estimates obtained from noisy function observations. SG algorithms that update without direct (though possibly noisy) observations of the function gradient are referred to as zeroth order SO algorithms. The first such algorithm was presented in \cite{kiefer}, and requires $2d$ function observations per iteration since it perturbs each parameter component separately. 

Such an approach does not scale well in terms of the computational complexity as the parameter dimension $d$ is increased. 
 
 The random directions stochastic approximation (RDSA) method has been presented in \cite{kushcla} and also explored in \cite{chin1997comparative}. The idea here is to randomly perturb all parameter components simultaneously using random vectors that are uniformly distributed over the surface of the unit sphere. In \cite{prashanth2017rdsa}, independent, symmetric, uniformly distributed perturbations have been explored and both gradient and Newton RDSA algorithms have been proposed and analyzed for their asymptotic convergence properties and rates.
 
 The smoothed functional (SF) algorithm based on independent Gaussian random perturbations has been independently studied in \cite{katkul,Kreimer,nesterov17,bhatnagar}. We shall refer to this algorithm as GSF. The idea underlying GSF is to approximate the convolution of the objective-function gradient with a multi-variate Gaussian PDF with the convolution of the objective function with a scaled Gaussian. Thus, this procedure works with only one simulation regardless of the parameter dimension. A two-simulation finite-difference version of the same with lower bias has been studied in \cite{STYBLINSKI90} and \cite{chin1997comparative}. In \cite{bhat2}, Newton-based algorithms with biased gradient and Hessian estimates obtained from Gaussian perturbations have been analyzed for their asymptotic convergence.

 The simultaneous perturbation stochastic approximation (SPSA) algorithm, see \cite{Spall92}, has also gathered attention due to the low computational effort required in this scheme as well as the ease of implementation. This algorithm randomly perturbs all parameter components simultaneously by using perturbation random variates whose properties are commonly satisfied by independent, symmetric, zero-mean, Bernoulli r.v.

\vspace{-2.5ex}

\paragraph{Our contributions.} We now summarize our contributions below.\\
{\textit{(a) SG with Cauchy perturbations:}} For solving the SO problem \cref{def:eq1}, we propose an SG algorithm, which performs gradient estimation using SF estimates based on $d$-dimensional truncated Cauchy perturbations inside a sphear of radius $\delta$. 
\\[0.5ex]
{\textit{(b) Asymptotic convergence:}} We prove that our algorithm converges asymptotically to the set of local minima of the objective function $f$. 
Note here that one can ordinarily prove convergence of an SG algorithm to the set of equilibria of the associated gradient ODE. These points however also include local maxima and saddle points (in addition to local minima) of the given objective function. We however prove by verifying a result from \cite{pemantle} that the convergence of our algorithm is only to local minima. In fact, the other fixed points (that are not local minima) are unstable equilibria of the ODE and which we show are avoided by the algorithm in the limit. A result of this nature  has not been claimed for GSF, to the best of our knowledge. \\[0.5ex]
{\textit{(c) Asymptotic convergence rate:}} Our algorithm provides better asymptotic mean-squared error (AMSE) as compared to SPSA and GSF, which are two very popular schemes for gradient estimation in the zeroth-order SO context that we consider. AMSE is a standard metric for comparing the (asymptotic) convergence rate of different algorthms, cf. \cite{chin1997comparative,prashanth2017rdsa}, and a better AMSE is beneficial in simulation optimization applications, where each function measurement is assumed to be computationally intensive. 
\\[0.5ex]
{\textit{(d) Non-asymptotic bound:}} We also provide  non-asymptotic bounds in the spirit of \cite{gadimi}, i.e., to an $\epsilon$-solution (see Definition \ref{def:def9} below) of the SO problem mentioned above. In the latter work, the authors consider a GSF scheme for gradient estimation, and provide an $O\left(\frac{1}{\epsilon^2}\right)$ bound on the number of iterations to find an $\epsilon$-solution under the assumption that the function $F$ accounting for the noisy observations is smooth. We provide a matching bound for our proposed algorithm. Further, unlike \cite{gadimi}, we also provide non-asymptotic bounds in the case when $F$  is not assumed to be smooth.
\\[0.5ex]
{\textit{(e) Simulation experiments:}} Numerical results using a quadratic function, Rastrigin's function and Rosenbrock's function establish that our algorithm outperforms GSF, SPSA and RDSA algorithm.
  
\paragraph{Comparison to related works.} In \cite{chin1997comparative,Spall92,bhatnagar03, prashanth2017rdsa}, the authors employ random  perturbations based gradient estimates within an SG framework, and mainly show asymptotic convergence to a stationary point for their algorithms.  
Our algorithm, on the other hand, is shown to converge asymptotically to local minima, and more importantly, at a better rate as our algorithm possesses a better AMSE in comparison to the aforementioned algorithms.

SO algorithms with non-convex objectives invariably suffer from the problem of converging to stationary points that are not necessarily local minima.  \cite{Rong15} suggests adding an additional noise term in the gradient estimate when entering in the neighborhood of a stationary point that would ensure escape from saddle points with high probability while working with an unbiased SG estimate. In this work, however, we show that the inherent noise in the biased stochastic gradient estimate helps one escape unstable equilibria such as saddle points and local maxima without injecting additional noise. Another approach tried out is to use a second-order method \cite{Zhu,krishna} to escape from unstable equilibria, but second-order methods in general are more computationally expensive.

We now briefly review some prior work on non-asymptotic analyses in the setting of a biased gradient oracle, which encapsulates the properties of a gradient estimator based on SF. In \cite{Base09,Devolder14}, algorithms with a deterministic gradient oracle are presented. In \cite{Prashanth16}, the authors analyze the rates achievable with inputs from a biased stochastic gradient oracle, under a convexity assumption. In \cite{Bhavsar2022}, the authors analyze SG algorithms for solving a non-convex SO problem with inputs from a biased noisy gradient oracle. The rate that they derive for a smooth objective matches the bound for our algorithm with a balanced estimator. Further, unlike \cite{Bhavsar2022}, we derive an improved non-asymptotic bound under a smoothness assumption for the noisy observation $F$. This rate is $\mathcal{O}(1/\sqrt{N})$, where $N$ is the number of iterations of our algorithm, and it matches the bound obtained for GSF in \cite{gadimi}. In \cite{krishna}, the authors derive a non-asymptotic bound of $\mathcal{O}(1/\sqrt{N})$ for a zeroth-order variant of the stochastic conditional gradient algorithm using Gaussian perturbations. We derive a matching bound for our algorithm, which is more efficient than the one in \cite{krishna} since their algorithm requires solving an optimization problem in each iteration.

The rest of the paper is organized as follows:
The framework for the optimization problem and some preliminaries on GSF and TCSF are presented in \cref{section-2}. \cref{section-3} provides the main asymptotic and non-asymptotic guarantees for TCSF while the convergence analysis is discussed in \cref{section-4}. \cref{experiments} presents simulation experiments that compare the performance of TCSF with several algorithm.

\section{SF gradient estimation with truncated Cauchy perturbations}
\label{section-2}
We now present the idea behind a SF scheme for estimating the objective function gradient proposed in \cite{katkul}.\\ 
For a function $f:\mathbb{R}^d\to\mathbb{R}$, define the smoothed function $g_\delta:\mathbb{R}^d\to\mathbb{R}$ as
\begin{align}
\label{eq1}
    g_\delta(x)&=\E_{h_\delta(u)} [f(x+u)]  =\E_{h_\delta(x-u)} [f(u)], \mbox{ } x\in \mathbb{R}^d,
\end{align}
where $h_\delta(x)$ is called the smoothing kernel or perturbation density function. 
The parameter $\delta$ controls the degree of smoothness of $g_\delta(x)$. \\
The SF scheme is useful especially if $f(x)$ is not well behaved, for instance, if it has several stationary points in a narrow region. In such a case, one may work with $g_\delta(x)$ as it would exhibit a smoother behavior, and in general, it would be easier to compute it's derivative as opposed to $f(x)$. %As a result, it is important for $g_\delta(x)$ to be differentiable while retaining the favourable qualities of $f(\cdot)$. 
We have listed down the conditions from \cite{rubinstein} which suggests that a density function needs to satisfy all of them to become a smoothed function in the SF scheme.

\subsubsection*{Rubinstein's Conditions for SF schemes}
\label{Rub}
\begin{enumerate}[label=(\alph*)]
    \item $h_\delta:\mathbb{R}^d\to\mathbb{R}$ such that $h_\delta(u)=\frac{1}{\delta^d}h(\frac{u}{\delta})$ is piecewise differentiable with respect to  $u$,
    % \item $h_\delta(u)$ is piecewise differentiable in $u$,
    \item $h_\delta(u)$  is a probability density function such that  $g_\delta(x) =\E_{h_\delta(u)}[f(x+u)]$ ,
    \item $\lim_{\delta \to 0}h_\delta(u)=\Delta(u)$, where $\Delta(u)$ is the Dirac-Delta function,
    \item $\lim_{\delta \to 0}g_\delta(\cdot)=f(\cdot)$.
\end{enumerate}

\begin{definition}
\label{def:def1}
{\bf [Truncated Cauchy Distribution]} A r.v.~is said to follow the truncated (to the $\delta$-sphere) Cauchy distribution with mean vector zero and covariance matrix $\Sigma = \delta^2\mathbb{I}_{d\times d}$ if its probability density function has the following form (with
$c_1$ being the normalization constant):
\begin{equation}
\label{def-cauchy}
    h_\delta(u)=\frac{\Gamma(\frac{d+1}{2})}{\pi^{\frac{d+1}{2}}c_1\delta^d}\frac{1}{(1+\frac{\lVert u \rVert^2}{\delta^2})^{\frac{d+1}{2}}} \:\:\ \mbox{ for }\: \lVert u \rVert ^2\leq \delta^2.
\end{equation}
\end{definition}
In the following remark we have shown how the truncated Cauchy distribution, defined above, satisfies the aforementioned conditions. 
\begin{remark}
We now show that the truncated Cauchy distribution that we employ satisfies the Rubinstein's conditions (a)-(d) stated above. Hence, it is a valid candidate to be used as a smoothing density functional in the SF algorithm.\\
Note that $h_\delta$ in our case is a truncated Cauchy distribution as defined in \cref{def-cauchy}
which is a probability density function. It is easy to see that $h_\delta(u)=$ $\displaystyle\frac{1}{\delta^d}h(\frac{u}{\delta})$ where  \[\displaystyle h(\frac{u}{\delta})=\frac{\Gamma(\frac{d+1}{2})}{\pi^{\frac{d+1}{2}}c_1}\frac{1}{(1+\frac{\lVert u \rVert^2}{\delta^2})^{\frac{d+1}{2}}}.\] Hence $h(u)$ denotes truncated Cauchy distribution  with $\delta=1$. However from \cref{eq1} using this $h_\delta(u)$ one can write $g_\delta$ as expectation of $f$ under $h_\delta(u)$ i.e., $g_\delta(x)=\E_{h_\delta(u)}[f(x+u)]$. We know that Dirac-Delta function ($\Delta(u)$) is a measure whose value is $\infty$ at origin and $0$ otherwise with $\displaystyle \int_{-\infty}^{\infty} \Delta(u) du=1$. For truncated Cauchy distribution, $\lim_{\delta \to 0}h_\delta(u)=\infty$ which is a Delta function and $\lim_{\delta \to 0}g_\delta(\cdot)=f(\cdot)$ as $\int_{\lVert u \rVert^2 \leq 1} \Delta(u) du=1$. Thus, Rubinstein’s conditions are satisfied in the case of the truncated Cauchy distribution.
\end{remark}
A truncated Cauchy distribution has been considered in Chapter 6 of \cite{bhatnagar}. However, the truncation there is not to the $\delta$-sphere, and more importantly, unlike them, we derive asymptotic as well as non-asymptotic rate results with the truncated Cauchy distribution specified above.

One can intuitively interpret the effect of $\delta$ on smoothing as follows: For smaller values of $\delta$, $g_\delta(x)$ is close to $f$. However, as $\delta$ increases, $g_\delta(x)$ becomes smoother with fewer sharp variations.
As explained in \cite{rubinstein}, the SF approach provides a helpful way for approximating the gradient of any function $f$. In particular, we have shown in the following proposition that for truncated Cauchy smoothing, the derivative of $g_\delta(x)$ can be calculated by taking the derivative of $h_\delta(x-u)$. This can be obtained via a simple application of the Leibnitz rule and  piece-wise differentiation property of $h_\delta(u)$.

\begin{proposition}
\label{pf7}
Let $g_\delta(x)$ be the smoothed function defined in \cref{def:eq1} under truncated Cauchy distribution. Then gradient of $g_\delta$ is
\begin{equation}
\label{eq2}
    \begin{split}
    \nabla g_\delta(x) =\frac{1}{\delta} \mathbb{E}_{h(u)} \left[f(x+\delta u) 
        \frac{(d+1)u}{(1+\lVert u \rVert^2)}\right].\\
    \end{split}
\end{equation}
\end{proposition}
\begin{proof}
\cref{eq1} can be rewritten by considering  $u = \delta v$ as follows:
\begin{equation*}
    \begin{split}
        g_\delta(x)&=\frac{\Gamma(\frac{d+1}{2})}{\pi^{\frac{d+1}{2}}c_1}\int_{\left\lVert v\right\rVert^2\leq 1} \frac{f(x+\delta v)}{(1+\left\lVert v\right\rVert^2)^{\frac{d+1}{2}}} dv.\\
     \end{split}
\end{equation*}
Let's rewrite the above equality by a change of
variable $y=x+\delta v$ as follows:
\begin{equation*}
    \begin{split}
        g_\delta(x)&=\frac{\Gamma(\frac{d+1}{2})}{\delta^d\pi^{\frac{d+1}{2}}c_1}\int_{\left\lVert \frac{y-x}{\delta}\right\rVert^2\leq 1} \frac{f(y)}{(1+\left\lVert \frac{y-x}{\delta}\right\rVert^2)^{\frac{d+1}{2}}} dy.\\
     \end{split}
\end{equation*}
We apply the classic rule of differentiation over $x$ in the above 
equation, and simplify as follows:
\begin{equation*}
    \begin{split}
    \nabla g_\delta(x)
     &=\frac{1}{\delta}\int_{\lVert u \rVert^2\leq 1} f(x+\delta u )\bigg(\frac{(d+1)u}{(1+\lVert u \rVert^2}\bigg) h(u) du
      =\frac{1}{\delta }\E_u\bigg[\frac{f(x+\delta u)(d+1)u}{(1+\lVert u \rVert^2)}\bigg],
    \end{split}
\end{equation*}
where $\E_u$ denotes expectation w.r.t. $h(u)$.
\end{proof}

In \cref{eq2}, $h(u)$ indicates that $h_\delta$ is computed at $\delta=1$. Notice that the one-simulation smooth function under GSF scheme is specified in \cite{rubinstein} by
$\E_{h(u)}[f(x+\delta u)]$ and its derivative can be approximated by
${\displaystyle \frac{1}{\delta}\E_{h(u)}[f(x+\delta u)u]}$ where $u$ has the multivariate Gaussian distribution, a form also studied in \cite{bhatnagar03} for the case when the objective function is a long-run average cost. A finite-difference form of $\nabla g_\delta(x)$ would be ${\displaystyle \frac{1}{\delta}\E_{h(u)}[(f(x+\delta u)-f(x))u]}$. 

We now define the $\delta$-difference smoothed function $f_\delta(x):\R^d\to \R$ for $f(x)$ (an approximation to smoothed function) as below:
\begin{equation}
\label{def:def3}
    \begin{split}
        f_\delta(x)&=\E_{h_\delta(u)} (f(x+ u)-f(x))=\E_{h_\delta(x-u)}[f(u)-f(x-u)],
        % &=c \int_{\left\lVert u\right\rVert^2\leq 1} (f(x+\delta u)-f(x)) \frac{1}{(1+\lVert u \rVert^2)^{\frac{d+1}{2}}} du\\
    \end{split}
\end{equation}
where $h_\delta(u)$ is smoothing kernel. The finite-difference gradient form $\nabla g_\delta(x)$ for GSF , see \cite{rubinstein}, is the same as $\nabla f_\delta(x)$ due to the fact that  $\E_{h(u)}(u)=0$ under Gaussian smoothing kernel.
As mentioned previously, we only have access to noisy observations of the objective function. Thus, to solve the problem \cref{def:eq1} using the SF algorithm one can consider the gradient of $f_\delta(x)$ to be (with $\xi$ having the same distribution as $\xi^+$) ${\displaystyle \nabla f_\delta(x)=\frac{1}{\delta}\E_{u,\xi}[(F(x+\delta u,\xi^+)-F(x,\xi))u]}$.  

It can similarly be shown as \cref{eq2} that for $f\in \mathcal{C}_L^{1,1}$ (satisfying \cref{def:ass2}) and with the truncated Cauchy smooth kernel, the gradient of $f_\delta$ can be expressed as follows:
% $\nabla f_\delta(x) =\frac{1}{\delta} \E_{h(u)} \left[(f(x+\delta u)-f(x))\frac{(d+1)u}{(1+\lVert u \rVert^2)}\right].$
\begin{equation}
\label{eq3}
    \begin{split}
    \nabla f_\delta(x) =\frac{1}{\delta} \E_{h(u)} \left[(f(x+\delta u)-f(x)) 
        \frac{(d+1)u}{(1+\lVert u \rVert^2)}\right].\\
    \end{split}
\end{equation}

\begin{algorithm}[t]
	\caption{Truncated Cauchy Smoothed Functional (TCSF) Algorithm}
	\label{alg:alg1}
	\begin{algorithmic}
		\STATE {\bfseries Input:} Initial point $x_1\in \R^d$,  non-negative step-sizes $\{\gamma_k\}_{k\geq1}$,  and smoothing parameter $\delta_k > 0$.
% 		probability distribution $\p_R(\cdot)$ over $\{1,2, \ldots, N\}$
		%  \REPEAT
		%   \STATE Initialize $noChange = true$.
		\FOR{$k=1,2,\ldots$}
		\STATE Generate $u_k$ from \cref{def-cauchy},
		 compute $G(x_k,\xi_k^+,\xi_k,u_k,\delta_k)$ using \cref{eq4}, %given by
		%$$\bigg(\frac{F(x_k+\delta %u_k,\xi_k^+)-F(x_k,\xi_k)}{\delta}\bigg)\frac{(d+1)u_k}{1+\lVert u_k %\rVert^2},$$
		and update
		\begin{equation}
			\label{def:eq2}
			x_{k+1}=x_{k}-\gamma_k G(x_k,\xi_k^+,\xi_k,u_k,\delta_k).
		\end{equation}
		\ENDFOR
		%  \UNTIL{}
	\end{algorithmic}
\end{algorithm}

In this work, we propose the $\delta$-difference Cauchy smoothed functional scheme. Now note that,\\ $\E_{h(u)}\left[\frac{u}{1+\lVert u \rVert^2}\right]\ne 0$ and so we propose a two-simulation finite-difference gradient estimate instead of a one-simulation estimate as it can be seen to have a lower bias. 

For the case of noisy function observations, see \eqref{def:eq1}, the SF gradient with truncated Cauchy would simply be ${\displaystyle \nabla f_\delta(x)=}$ ${\displaystyle \frac{1}{\delta}\E_{u,\xi^+,\xi}\left[(F(x+\delta u,\xi^+)-F(x,\xi))\frac{(d+1)u}{(1+\lVert u \rVert^2)}\right]}$ and a one-sample gradient estimate would have the form
${\displaystyle 
    \nabla f(x) \approx}$
    ${\displaystyle  \bigg(\frac{F(x+\delta u,\xi^+)-F(x,\xi)}{\delta}\bigg)
        \frac{(d+1)u}{(1+\lVert u \rVert^2)},
}$
for $\delta>0$ small. 
When this estimate is used in a stochastic approximation (SA) scheme, averaging would happen naturally. 
Thus, under Cauchy perturbation, one can consider the estimate of $\nabla f(x_k)$ for a given parameter value $x_k$ (that in turn would be updated as per an SA scheme as mentioned above) to be as follows:

\begin{flalign}
\label{eq4}
    G(x_k,\xi_k^+,\xi_k,u_k,\delta_k)
&\stackrel{\triangle}{=} {\displaystyle 
\bigg(\frac{F(x_k+\delta_k u_k,\xi_k^+)-F(x_k,\xi_k)}{\delta_k}\bigg)\frac{(d+1)u_k}{1+\lVert u_k \rVert^2}}.
\end{flalign}
In the above estimates, $\{\xi_k\}$ and $\{\xi_k^+\}$ constitute the measurement noise r.vs in the two simulations and are assumed to be i.i.d., having a common distribution, and further independent of one another.
\cref{alg:alg1} contains the details of the update procedure.
\paragraph{Motivation for truncation in Cauchy perturbations:} 
To get a bias bound for $G$ in \eqref{eq4}, one needs the expectation of $\frac{f(x+\delta_k u_k)-f(x_k)}{\delta_k}\frac{(d+1)u_k}{1+\lVert u_k \rVert^2}+\frac{\eta_k^+-\eta_k}{\delta_k}\frac{(d+1)u_k}{1+\lVert u_k \rVert^2}$, which can be obtained after a Taylor series expansion. However, this expectation only exists if the mean and covariance matrix of the perturbation distribution exist (see proof of \cref{def:lem2} in \cref{sec:appendix-asymp}). Thus we incorporate truncation in the random perturbation.

\section{Convergence Results}
\label{section-3}
\subsection{Asymptotic convergence}
Let $\mathcal{F}_k=\sigma(x_m, m\leq k; u_m, \xi_m^+, \xi_m, m < k)$, $k\geq 1$, denote a sequence of increasing sigma fields. Let $\eta_k^+=F(x_k+\delta_k u_k,\xi_k^+)-f(x_k+\delta_k u_k),\;\;
        \eta_k=F(x_k,\xi_k)-f(x_k)$,
respectively, where $\{\delta_k\}$ is a sequence of smoothing parameters that diminishes to zero as $k\to \infty$. It is easy to see that 
$\E[\eta_k^+-\eta_k|\mathcal{F}_k]=0$. 

We now make the following assumptions as in \cite{bhatnagar}.

\begin{assumption}
\label{def:ass3}
The step-size $\gamma_k$ and the smoothing parameter $\delta_k$ are positive for all $k$. Further, $\gamma_k, \delta_k \rightarrow 0$ as $k\to \infty$ and 
$\sum_k \gamma_k=\infty ,\ \sum_k\left(\frac{\gamma_k}{\delta_k}\right)^2 < \infty.$
\end{assumption}
\begin{assumption}
\label{def:ass4}
The function $f$ is three-times continuously differentiable with $\rVert\nabla^2 f(x)\lVert\leq B$ and $\vert \nabla^3_{i,j,k} f(x) \vert\leq B_1$ for $i,j,k=1,\ldots ,N$ where $B,B_1\geq0$

\end{assumption}
\begin{assumption}
\label{def:ass5}
There exist $\beta_1,\beta_2>0$ such that $\forall k \geq 0$, $\E\vert \eta_k \vert^2\leq \beta_1, \E\vert \eta_k^+ \vert^2\leq \beta_1 ,\ \E|f(x_k\pm\delta_ku_k)^2|\leq \beta_2$ and $\E|f(x_k)^2|\leq \beta_2$. 
\end{assumption}
\begin{assumption}
\label{def:ass6}
$\sup_k\lVert x_k \rVert<\infty$ almost surely.
\end{assumption}

The above assumptions are commonly used in the convergence analysis of an SA algorithm. In particular,
\cref{def:ass3} is a standard SA requirement on the  step-sizes.
 \cref{def:ass4} ensures that the associated ODE is well-posed and helps in establishing the asymptotic unbiasedness of the estimated gradient.
The conditions in \cref{def:ass5} ensure that the effect of noise can be ignored in asymptotic analysis of \cref{def:eq2}. 
\cref{def:ass6} is a stability assumption that ensures convergence of \cref{def:eq2} and is common to the analysis of simultaneous perturbation based SG algorithm cf.~ \cite{Spall92,bhatnagar}. If stability is hard to ensure, a common practice is to project the iterate sequence onto a compact and convex set, and such a scheme would fall under the realm of projected stochastic approximation.   

The following lemma characterizes the relationship between the estimator $G$ and the true gradient of the objective $f$.
\begin{lemma}
\label{def:lem2}
Under \ref{def:ass3}-\ref{def:ass6}, %and $G_\delta$ defined in \cref{alg:alg1} 
we have almost surely
\begin{align}
	\label{eq:grad-bias}
        \E[G(x_k,\xi_k^+,\xi_k,u_k,\delta_k)|\mathcal{F}_k]= c_2\nabla f(x_k)+\delta_k w_k,
\end{align}
where $c_2 = \E\left[\frac{(d+1)(u_k^1)^2}{1+\lVert u_k \rVert^2}\right]$, with $u_k^1$ denoting the first element of the random vector $u_k$, and \\$w_k=\E\left[\bigg(\frac{ u_k^T\nabla^2 f(\Bar{x}^+_k) u_k}{2}\bigg)\left.\frac{(d+1)u_k}{1+\lVert u_k \rVert^2}\right|\mathcal{F}_k\right]$. 
%\mathcal{O}(\delta_k)$ can be treated as a vector with each element multiplied by $\delta_k$
\end{lemma}
\begin{proof}
See \cref{pf1}.
\end{proof}
Lemma \ref{def:lem2} does not bound the bias in the gradient estimator directly. This is because of the constant factor $c_2$ multiplying the gradient term on the RHS of \eqref{eq:grad-bias}. However, the result in \eqref{eq:grad-bias} is useful in establishing asymptotic convergence of Algorithm \ref{alg:alg1} as it tracks the following ODE 
\begin{equation}
\label{ode}
    \dot{x(t)} = -c_2\nabla f(x(t)).
\end{equation}
In fact for $c_2>1$, we will obtain faster convergence, 
see  \cref{eq:5} in \cref{sec:appendix-asymp} for a detailed argument.
\begin{theorem}
\label{th1}(Strong Convergence):
Assume \cref{def:ass3}-\cref{def:ass6} hold. The sequence $x_k,k\geq 1$ governed by \cref{def:eq2} satisfies 
\[x_k \to H\stackrel{\triangle}{=} \{x^*|\nabla f(x^*)=0\} \textrm{ a.s. as } k\to \infty.\]
\end{theorem}
\begin{proof}
See \cref{pf2}.
\end{proof}
The following variant of \cref{def:ass5} is required to prove the asymptotic normality of $x_k$.
\begin{assumption}
\label{def:ass8}
Assume  \cref{def:ass5} holds. In addition, there exists $\sigma'>0$  such that $\E(\eta_k^+-\eta_k)^2\to \sigma'^2$ a.s. as $k \to\infty$.
\end{assumption}
\begin{theorem}
\label{th2}
(Asymptotic Normality):
Assume \cref{def:ass3}--\cref{def:ass8} hold. Also, assume that the set $H$ in \cref{th1} is a singleton $H=\{x^*\}$. Let $\gamma_k=\gamma_0/k^\alpha$ and $\delta_k=\delta_0/k^\phi$, where $\gamma_0,\delta_0>0,\alpha\in(0,1]$ and $\phi\geq \alpha/6$. Furthermore, let $\upsilon=\alpha-2\phi>0$ and $Q$ be the orthogonal matrix with $Q\nabla^2f(x^*)Q^T=\gamma_0^{-1}$diag$(\lambda_1,\lambda_2,\ldots,\lambda_d)$, with $\lambda_1,\ldots,\lambda_d$ being the eigen-values of $\nabla^2f(x^*)$. Then,
$k^{\upsilon/2}(x_k-x^*)\stackrel{dist}{\to}\N(\mu,QMQ^T)$ as $k\to \infty$ a.s.,
where $\N(\cdot,\cdot)$ denotes a multi-variate Gaussian distribution with mean $\mu$ defined by
\[
    \mu=
    \begin{cases}
        0 & \text{if}\ 3\phi-\alpha/2>0,\\
        \Bar{c}(\gamma_0\delta_0^2(\gamma_0\nabla^2f(x^*)-\frac{1}{2}\upsilon^+\mathbb{I})^{-1}T) & \text{if}\ 3\phi-\alpha/2=0. 
    \end{cases}
    \]In the above, $\mathbb{I}$ is identity matrix,$\Bar{c}=\E[\lVert u \rVert^4]$ , 
    $ \upsilon^+=\upsilon$ if $\alpha=1$ and $0$ if $\alpha<1,\ T$ is a d-dimensional vector whose $i$-th element given by
$-\frac{1}{6}[\nabla^3_{iii}f(x^*)+3\sum_{j=1,j\neq i}\nabla^3_{jji}f(x^*)]$ and the matrix is defined as
$M=\frac{\gamma^2_0\sigma'^2}{4\delta^2_0}diag((2\lambda_1-\upsilon^+)^{-1},\ldots,(2\lambda_d-\upsilon^+)^{-1}).$ 
\end{theorem}
\begin{proof}
See \cref{pf:th2}.
\end{proof}
\begin{remark}
\label{Hstar}
For the case of general $H$ (not necessarily singleton as in the statement of \cref{th2}, the result will continue to hold for the particular $x^*$ in whose neighborhood the parameter lies after a sufficiently large number of iterations.
\end{remark}
From \cref{th2} we can say that $ k^{\upsilon/2}(x_k-x^*)$ is asymptotically Gaussian for TCSF algorithm. In addition, the maximum possible value of $\upsilon=2/3$ can be obtained by fixing $\alpha=1$ and $\phi=1/6$ in $\upsilon=\alpha-2\phi$. Fixing $\upsilon=2/3$, we obtain the best possible asymptotic convergence rate of $k^{-1/3}$. 
% For general $H$ (not necessarily singleton), see \cref{Hstar} in Appendix.
% In addition, by the relation between $\upsilon$ and $\alpha,\phi$ i.e., $\upsilon=\alpha-2\phi$ it is easy to observe that $\upsilon$ lies between $(0,2/3]$. 
We define the asymptotic mean square error (AMSE), cf.~\cite{Spall92}, by
%\begin{equation*}
    $AMSE(\gamma_0,\delta_0)=\mu^T\mu+trace(QMQ^T)$,
%\end{equation*}
where $\gamma_0,\delta_0$ are constant step-size and smoothing constant respectively. It can be shown under the condition given in \cite{gerencser99} that $AMSE(\gamma_0,\delta_0)$ coincides with $k^\upsilon\E\lVert x_k-x^*\rVert^2$.

We now compare the AMSE of our algorithm with two other well-known random perturbation gradient estimation schemes, namely GSF and SPSA 

To make the comparison fair, we follow the approach of  \cite{chin1997comparative} and set $\gamma_k$ and $\delta_k$ uniformly. 
In particular,  we use step-size $\gamma_k=\gamma_0/k$, where $\gamma_0\geq \upsilon/2\lambda_0$, with $\lambda_0$ denoting the minimum eigenvalue of $\nabla^2 f(x^*)$. Further, we set $\delta_k=\frac{\delta_0}{k^{1/6}}$. It can then be shown that
\begin{equation}
    AMSE(\gamma_0,\delta_0)=\Big(\bar{c}\delta^2_0\gamma_0\lVert\Phi T\rVert\Big)^2+\frac{1}{\delta^2_0}trace(\Phi P),\label{eq:amse-general}
\end{equation}
where  $T$ is as defined in \cref{th2}, 
$\Phi=(\gamma_0\nabla^2f(x^*)-\frac{1}{2}\upsilon^+\mathbb{I})^{-1}$ and $P=\frac{\sigma'^2}{4}\mathbb{I}$. 

\begin{remark}[Comparing with GSF]
% For $u_k$ follows $\N(0,\I)$ we have $\E(u_k^i)^4=3$. The ratio of AMSE of GSF with TCSF can be obtained as
% $\E(u_k^i)^4=3$, where $u_k \sim \N(0,\I), \forall k$. The ratio of AMSE of GSF with our TCSF algorithm is given by
% \[\frac{AMSE_{GSF}(\gamma_0,\delta_0)}{AMSE_{TCSF}(\gamma_0,\delta_0)}=\frac{\Big(3\delta^2_0\gamma_0\lVert\Phi T\rVert\Big)^2+\frac{1}{\delta^2_0}trace(\Phi P)}{\Big(\frac{c_{11}\delta^2_0\gamma_0}{d+2}\lVert\Phi T\rVert\Big)^2+\frac{1}{\delta^2_0}trace(\Phi P)},\]
% where $\Phi,\ P$ $T$ are specified in \eqref{eq:amse-general}. For univariate Cauchy, it can shown that $c_1 <1$ using an expression from \cite{staneski}. For the general case of $d\ge 2$, it is hard to obtain an exact expression for $c_1$. From Monte Carlo simulations, we observed that $c_1$ lies in $[0.45,0.57]$ for various values of $d$ in ${2,\ldots,100}$. For this reason, we believe $c_{11}/(d+2)<1$, leading to a better AMSE for TCSF in comparison to GSF. 
For $u_k$ following $\N(0,\I)$, we have $\E(u_k^i)^4=3$. 
% The ratio of AMSE of GSF with TCSF can be obtained as
% $\E(u_k^i)^4=3$, where $u_k \sim \N(0,\I), \forall k$. 
Hence, the ratio of AMSE of GSF with that of TCSF algorithm is given by
\begin{flalign*}
\frac{AMSE_{GSF}(\gamma_0,\delta_0)}{AMSE_{TCSF}(\gamma_0,\delta_0)}
&=\frac{\Big(3\delta^2_0\gamma_0\lVert\Phi T\rVert\Big)^2+\frac{1}{\delta^2_0}trace(\Phi P)}{\Big(\Bar{c}\delta^2_0\gamma_0\lVert\Phi T\rVert\Big)^2+\frac{1}{\delta^2_0}trace(\Phi P)}.
\end{flalign*}
where $\Phi,\ P$ $T$ are specified in \eqref{eq:amse-general}. 
Recall that $\Bar c = \E[\lVert u \rVert^4]$, and $u$ is restricted to the unit sphere, implying $\Bar c \le 1$. Thus, the AMSE of TCSF is clearly better than that of GSF.
\end{remark}

\begin{remark}[Comparing with SPSA]
% Comparing the AMSE of TCSF with SPSA, we obtain
% \[\frac{AMSE_{SPSA}(\gamma_0,\delta_0)}{AMSE_{TCSF}(\gamma_0,\delta_0)}=\frac{\Big(\delta^2_0\gamma_0\lVert\Phi T\rVert\Big)^2+\frac{1}{\delta^2_0}trace(\Phi P)}{\Big(\frac{c_{11}\delta^2_0\gamma_0}{d+2}\lVert\Phi T\rVert\Big)^2+\frac{1}{\delta^2_0}trace(\Phi P)}.\]
% As in the case of GSF, the AMSE of TCSF would be better than that of SPSA if $c_{11}/(d+2)<1$, which we verified in theory for $d=1$, and in practice for $d\ge 2$.
Comparing the AMSE of TCSF with SPSA, we obtain
\begin{flalign*}
\frac{AMSE_{SPSA}(\gamma_0,\delta_0)}{AMSE_{TCSF}(\gamma_0,\delta_0)}
&=\frac{\Big(\delta^2_0\gamma_0\lVert\Phi T\rVert\Big)^2+\frac{1}{\delta^2_0}trace(\Phi P)}{\Big(\Bar{c}\delta^2_0\gamma_0\lVert\Phi T\rVert\Big)^2+\frac{1}{\delta^2_0}trace(\Phi P)} \ge 1.
\end{flalign*}
The last inequality follows from the fact that $\Bar c \le 1$. 
Thus, the AMSE of TCSF is at least as good as that of SPSA, and would be better if $\Bar c<1$.  For the case of $d=1$, $\bar c < 1$, as shown in \cite{staneski}. On the other hand, it is difficult to obtain a closed-form expression for $\Bar c$ when $d>1$.
\end{remark}

From the analysis above, it is apparent that, from an asymptotic convergence rate viewpoint, our algorithm outperforms GSF and SPSA, which are two popular gradient estimation schemes. \\
So far, we have provided theoretical guarantees that establish convergence to a stationary point of the objective function $f$. However, this result is not sufficient in a non-convex optimization setting since local maxima and saddle points are also stationary points in addition to local minima. The aim here is to converge to a local minimum, or avoid traps such as saddle points/local maxima. A usual trick to achieve this objective is to add extraneous noise, so that the algorithm does not converge to an unstable equilibrium. We now establish that our algorithm naturally avoids traps owing to the noise in the gradient estimator. For this result, we require the following additional assumption.
%We now state a result from \cite{pemantle}.
\begin{assumption}
\label{def:ass9}
Assume the condition in \cref{def:ass8} holds and $c_9$ is such that $\E|\eta_k^+-\eta_k|\geq c_9$.
\end{assumption}
The assumption above ensures that the noise is rich in all directions. Note that, one can rewrite the update rule \cref{def:eq2} as
\begin{equation}
\label{def:eq3}
x_{k+1} =x_{k}-\gamma_kG (x_k,u_k,\delta_k)-\epsilon_k,
\end{equation}
with
${\displaystyle G (x_k,u_k,\delta_k)=\frac{f(x_k+\delta_k u_k)-f(x_k)}{\delta_k}\frac{(d+1)u_k}{1+\lVert u_k \rVert^2}}$ and ${\displaystyle \epsilon_k=\gamma_k\frac{\eta_k^+-\eta_k}{\delta_k}\frac{(d+1)u_k}{1+\lVert u_k \rVert^2}}$.\\
Here $\E[\epsilon_k|\mathcal{F}_k]=0$. Note that \cref{def:eq3} is equivalent to \cref{def:eq4} in \cref{sec:appendix-asymp} by considering $-\epsilon_k=\mu_k$ and $-G =Y$. 
% It is worth considering \cref{def:ass9} as it helps to get a lower bound on the noise. 
One can  ensure avoidance of traps if the increment of $\epsilon_k$ in any direction is of order $1/n^\gamma$, i.e., $\displaystyle\E\left[(\epsilon_k \cdot \theta)^+|\mathcal{F}_k\right]\geq c_7/k^\gamma$ (\cref{Them:1}) for every unit vector $\theta$. We establish that our algorithm satisfies $\displaystyle\E\left[(\epsilon_k \cdot \theta)^+|\mathcal{F}_k\right]\geq c_9c_{10}/2$ with $c_9 = \mathcal{O}(1/k^\gamma)$ (see (\cref{pf3}) for details).
\begin{proposition}
\label{prop}
Under \cref{def:ass3}-\cref{def:ass9}, $x_k$ generated by \cref{alg:alg1}, converges to a local minimum a.s.
\end{proposition}
\begin{proof}
See \cref{pf3}.
\end{proof}
\begin{remark}
From \cref{prop} we can justify that \cref{def:eq2} avoids saddle points and local maxima. To the best of our knowledge, a similar result is not available for the GSF algorithm. The latter algorithm has been shown to converge to a stationary point in \cite{bhatnagar}, and a non-asymptotic convergence rate for the same is available in \cite{gadimi}. In \cite{Rong15}, \cite{chi17}, the authors suggest adding extraneous noise to avoid traps for a SG algorithm. In contrast, we show that a noisy gradient estimation scheme would naturally avoid traps, obviating the need for extraneous noise addition.
\end{remark}

\subsection{Non-asymptotic convergence}

The non-asymptotic analysis below establishes convergence to an approximate stationary point as with \cite{gadimi,Bhavsar2022}. 

\begin{definition}
\label{def:def9}
For a non-convex function $f$, $\bar{x}$ is said to be an $\epsilon$-stationary point to the  problem \cref{def:eq1} if it satisfies $\E\left[\left\lVert\nabla f(\bar{x})\right\rVert^2\right]\leq\epsilon$.
\end{definition}
For non-asymptotic analysis we make the following assumption.
\begin{assumption}
	\label{def:ass1}
	Let $ \sigma^2>0$ such that
	%$\E_\xi[F(x, \xi)]=f(x)$, 
 	$\E_\xi[\lVert \nabla F(x, \xi) - \nabla f(x)\rVert^2] \leq \sigma^2$,   
 	 for all $x\in \R^d$.
\end{assumption}
\begin{assumption}
\label{def:ass2}
The function $f:\R^d\to\R$ is differentiable and the gradient of $f$ is Lipschitz continuous, i.e., $\left\lVert\nabla f(x) - \nabla f(y) \right\rVert \leq L\left\lVert x-y\right\rVert$, $\forall x,y\in\R^d$, where $L>0$ is the  Lipschitz constant. 
\end{assumption}

We now present a non-asymptotic bound for a randomized variant of TCSF algorithm in the spirit of \cite{gadimi}. 
\begin{theorem}
\label{coro2}
Assume \cref{def:ass5}, \cref{def:ass1} and \cref{def:ass2}.
Suppose Algorithm \ref{alg:alg1} has the following as the step-sizes and smoothing parameters:
\[    \gamma_k \stackrel{\triangle}{=}\min\left\{\dfrac{c_2}{L},\dfrac{1}{N^{2/3}}\right\}\ , \
    \delta_k\stackrel{\triangle}{=}\delta=\dfrac{1}{N^{1/6}}, 
k=1,\ldots,N,\] where $c_2$ is as specified in Lemma \ref{def:lem2}. Let $x_R$ be picked uniformly at random from $\{x_1,\ldots,x_N\}$.
Then
\begin{equation*}
\begin{split}
    \E\left[\left\lVert\nabla f(x_R)\right\rVert^2\right]&\leq 2D\Big(\frac{L}{c_2N}+ \frac{1}{N^{1/3}}\Big)+\frac{2Bc_2'}{c_2N^{1/6}}+\frac{L}{2c_2}\left(\frac{2Bc_2'}{N^{5/6}} +\frac{dc_2'^2}{N}+\frac{C''}{N^{1/3}}\right),
\end{split}
\end{equation*}
where $C''=\E\left[\frac{\lVert u_k\rVert^2(d+1)^2}{(1+\lVert u_k\rVert^2)^2}\right]\big[2(\beta_1+\beta_2)\big]$ , $c_2'=(d+1)B$ and $B$ is as specified in \cref{def:ass4}.
\end{theorem}
% is as defined in \cref{martingle}
\begin{proof}
See \cref{pf:coro2}.
\end{proof}
From \cref{coro2}, we can conclude that the non-asymptotic bound is $\mathcal{O}(N^{-1/6})$, that is weaker than the asymptotic convergence rate of $\mathcal{O}(N^{-1/3})$ obtained in the previous section. This gap can however be fixed by using a balanced estimator, i.e.
\begin{equation}
\label{balanced}
    \begin{split}
        \tilde G(x_k,\xi_k^+,\xi_k,u_k,\delta_k)&
\stackrel{\triangle}{=} {\displaystyle 
\bigg(\frac{F(x_k+\delta_k u_k,\xi_k^+)-F(x_k-\delta_k u_k,\xi_k^-)}{2\delta_k}\bigg)\frac{(d+1)u_k}{1+\lVert u_k \rVert^2}}.
    \end{split}
\end{equation}
In this case, we obtain the following bound  (see \cref{prop3} in \cref{appendix-D} for details):
\[\E[\tilde G(x_k,\xi_k^+,\xi_k,u_k,\delta_k)|\mathcal{F}_k]=c_2\nabla f(x_k)+ \mathcal{O}(\delta^2_k) \mathbf{1}_{d}.\]
Notice that the bound is  $\mathcal{O}(\delta_k^2)$ instead of $\mathcal{O}(\delta_k)$ in the one-sided estimator \eqref{eq4}, since the second-order terms cancel in the balanced estimator. With a $\mathcal{O}(\delta_k^2)$ bias bound, we obtain in the following theorem a non-asymptotic bound of $\mathcal{O}(N^{-1/3})$ using the balanced estimator. From this bound it is observed that $\mathcal{O}(1/\epsilon^3)$ number of iterations of the TCSF algorithm guarantee convergence to an $\epsilon$-stationary point of \cref{def:eq1}.
\begin{theorem}
\label{prop4}
Assume \cref{def:ass5}, \cref{def:ass1} and \cref{def:ass2}.
Suppose Algorithm \ref{alg:alg1} is running with step-sizes $\gamma_k$ and $\tilde G(x_k,\xi_k^+,\xi_k,u_k,\delta_k)$ instead of $G(x_k,\xi_k^+,\xi_k,u_k,\delta_k)$. The smoothing parameter $\delta_k=\delta,\forall k$, and step-sizes $\gamma_k$ is chosen as defined in \cref{coro2}.
Let $x_R$ denote a point picked uniformly at random from $\{x_1,\ldots,x_N\}$.
Then under probability distribution \cref{eq98}
\begin{equation*}
\begin{split}
    \E\left[\left\lVert\nabla f(x_R)\right\rVert^2\right]&\leq \Big(\frac{2DL}{c_2N}+ \frac{2D}{N^{1/3}}\Big)+\frac{2Bc_2''}{c_2N^{1/3}}+\frac{L}{2c_2}\left(\frac{2Bc_2''}{N} +\frac{dc_2''^2}{N^{4/3}}+\frac{C''}{N^{1/3}}\right),
\end{split}
\end{equation*}
where $C''$ is as defined in \cref{martingle}, $c_2''$ is defined \cref{prop3} and $B$ is the same as in \cref{def:ass4}.
\end{theorem}
\begin{proof}
See \cref{pf:th4}
\end{proof}

For the non-asymptotic bounds presented above, we assumed that the objective function is smooth. Instead, if we make the stronger assumption that the noisy observation $F$ is smooth, then we can obtain a better non-asymptotic bound. 
We make this claim precise in the following.
\begin{assumption}
	\label{def:ass10}
	The function $F$ is Lipschitz continuous in the first argument, uniformly w.r.t the second, i.e., for any given $\xi$, $\left\lVert\nabla F(x,\xi) - \nabla F(y,\xi) \right\rVert \leq L\left\lVert x-y\right\rVert$, $\forall x,y\in \R^d$ almost surely.
\end{assumption}
Assumptions \cref{def:ass1} and \cref{def:ass10}  imply \cref{def:ass2},  since
\begin{equation*}
	\left\lVert\nabla f(x) - \nabla f(y) \right\rVert \leq \E_\xi \left\lVert\nabla F(x,\xi) - \nabla F(y,\xi) \right\rVert \leq L\left\lVert x-y\right\rVert,\;\; \forall x,y\in \R^d.
\end{equation*}
We consider now the following variant of the gradient estimator in \eqref{eq4}:
\begin{flalign}
\label{eq:grad-est-crn}
    G(x_k,\xi_k,u_k,\delta_k)
&\stackrel{\triangle}{=} {\displaystyle 
\bigg(\frac{F(x_k+\delta_k u_k,\xi_k)-F(x_k,\xi_k)}{\delta_k}\bigg)\frac{(d+1)u_k}{1+\lVert u_k \rVert^2}}.
\end{flalign}
Notice that both function observations in this case use the same noise factor $\xi_k$. Such a setting is possible when noise is added using common random numbers, for instance, in computer simulations. 
 In this setting, $G(x_k,\xi_k,u_k,\delta_k)$ defined in \eqref{eq:grad-est-crn} satisfies
\[\E_{\xi,h(u)}\left[G (x_k,\xi_k,u_k,\delta)\right] =\E_{h(u)}\left[\E_\xi \left[G (x_k,\xi_k,u_k,\delta)\right]\right] =\nabla f_\delta(x_k).\]

We now provide a non-asymptotic bound of the order $\mathcal{O}(N^{-1/2})$ for Algorithm \ref{alg:alg1} under the additional assumptions listed above.
\begin{theorem}
\label{coro1}
Assume \cref{def:ass1}, \cref{def:ass10}.
Suppose Algorithm \ref{alg:alg1} runs with step-sizes $\gamma_k$ and the smoothing parameter $\delta_k=\delta,\forall k$, chosen as follows: 
\[
\gamma_k =\min\left\{\frac{1}{2Lc_{13}},\frac{1}{c_{13}\sigma\sqrt{N}}\right\},\  \delta = \frac{1}{L\sqrt{dNc_{13}}},\]
where $c_{13} = \frac{4c_{11}c_{12}}{d+1}$, with $c_{11}$ defined in \cref{th2} and $c_{12}$ is the Frobenius norm of the generalized inverse of the matrix $\E_u\left(\frac{(d+1) uu^T}{1+\lVert u \rVert^2}\right)$. Let $x_R$ denote a point picked uniformly at random from $\{x_1,\ldots,x_N\}$.
Then
\[
    \E\left[\| \nabla f(x_R)\|^2\right] \leq \frac{c_{14}}{N}+\frac{2\sigma L}{\sqrt{N}},
\mbox{ where }  c_{14}=2L+c_{13}\sigma+8LBc_{13}^{3/2}+\frac{\sigma B c_{13}^{3/2}B}{L}+\frac{2c_{13}d+1}{L}. \]

\end{theorem}
\begin{proof}
See \cref{pf:coro1}.
\end{proof}

\begin{remark}
In \cref{coro1}, we notice that the convergence rate is $\mathcal{O}(1/\sqrt{N})$ or equivalently $\mathcal{O}(1/\epsilon^2)$ number of iterations are needed to find an $\epsilon$-stationary point of \cref{def:eq1}. This rate is better than the one obtained in \cref{coro2} and this is a consequence of \cref{def:ass10} which ensures $F$ is smooth.  
In \cite{gadimi}, for GSF, the number of iterations to find an $\epsilon$-stationary point is bounded by $\mathcal{O}(1/\epsilon^2)$, and our bound matches their result. The advantage with our algorithm is that it outperforms the GSF algorithm empirically. In the next section, we provide some examples to validate this claim. 
\end{remark}

We will provide in the following theorem a non-asymptotic bound of order $\mathcal{O}(N^-{1/2})$ for \cref{alg:alg1} with balanced estimator by assuming sample performance is smooth.

\begin{theorem}
\label{prop6}
Assume \cref{def:ass4}, \cref{def:ass1} and \cref{def:ass10}.
Suppose Algorithm \ref{alg:alg1} is running with $\tilde G(x_k,\xi_k^+,\xi_k,u_k,\delta_k)$ instead of $G(x_k,\xi_k^+,\xi_k,u_k,\delta_k)$. Let the smoothing parameter $\delta_k=\delta,\forall k$ and step-sizes $\gamma_k$ is chosen as follows:
\begin{equation}
\label{eq101}
\begin{split}
    \gamma_k =\min\left\{\frac{c_2}{2c^2_{11}L},\frac{1}{N^{1/2}}\right\} , \
    \delta = \frac{1}{N^{1/2}}.
\end{split}
\end{equation}
Here $c_2$ is same as in Lemma \ref{def:lem2}. Let $x_R$ denote a point picked uniformly at random from $\{x_1,\ldots,x_N\}$.
Then under probability distribution $\p_R(k)=\frac{1}{N}$
\begin{equation*}
\begin{split}
    \E\left[\left\lVert\nabla f(x_R)\right\rVert^2\right]&\leq  \Big(\frac{2DL}{Nc_2^2}+ \frac{2D}{N^{1/2}}\Big)+\frac{2Bc_2''}{c_2N}+\frac{L}{2c_2N^{1/2}}\left(4c_{11}\sigma^2+\frac{2d^2c_2''^2}{N^2}\right),
   \end{split}
\end{equation*}
where $c_2''$ is defined \cref{prop3} and $B$ is the same as in \cref{def:ass4}.
\end{theorem}
\begin{proof}
See \cref{pf:th6}
\end{proof}

\section{Convergence proofs}
\label{section-4}

\subsection{Asymptotic convergence proofs}
\label{sec:appendix-asymp}
\subsubsection{Proof of Lemma \ref{def:lem2}}
In the lemma below, we  state and prove a  bound on the moments of a truncated Cauchy r.v.
\begin{lemma}
\label{def:lem5}
Let u be a truncated Cauchy r.v., then for any positive integer r, we have 
\begin{equation}
\label{def:eq10}
    \E_u[\lVert u\rVert^{2r}]\leq \frac{c_{11}}{(r+d)},
\end{equation}
where $c_{11}=\frac{2\Gamma(\frac{d+1}{2})}{\sqrt{\pi}\Gamma(d/2)c_1}.$
\end{lemma}

\begin{proof}
\begin{equation*}
\begin{split}
     \mathbb{E}_u[\lVert u\rVert^{2r}]&=\frac{\Gamma(\frac{d+1}{2})}{\pi^{\frac{d+1}{2}}c_1}\int_{\lVert u \rVert^2\leq 1} \lVert u \rVert^{2r}\frac{1}{(1+\lVert u \rVert^2)^{\frac{d+1}{2}}} du \\
     &=\frac{\Gamma(\frac{d+1}{2})}{\pi^{\frac{d+1}{2}}c_1}\int_0^1\int_{\lVert u \rVert^2=R} R^{r}\frac{1}{(1+R)^{\frac{d+1}{2}}} du dR\\
     &=\frac{\Gamma(\frac{d+1}{2})}{\pi^{\frac{d+1}{2}}c_1}\int_0^1 R^{r}\frac{1}{(1+R)^{\frac{d+1}{2}}} \frac{2\pi^{d/2}}{\Gamma(d/2)}R^{d-1}dR\\
     &=c_{11}\int_0^1 \frac{R^{r+d-1}}{(1+R)^{\frac{d+1}{2}}}dR\\
     &\leq c_{11}\int_0^1 R^{r+d-1} dR\\
     &=\frac{c_{11}}{(r+d)}.
    \end{split}
\end{equation*}
\end{proof}
\subsubsection*{Proof of Lemma \ref{def:lem2}}
\label{pf1}
\begin{proof}
Using Taylor series expansions,
we obtain 
\begin{equation}
\begin{split}
    f(x_k+\delta_k u_k)&=f(x_k)+\delta_k u_k^T\nabla f(x_k)
    +\frac{\delta_k^2}{2}u_k^T\nabla^2 f(\Bar{x}^+_k)u_k, \label{eq:taylor}
\end{split}
\end{equation}
where $\Bar{x}^+_k$ is on the line segment between $x_k$ and $x_k+\delta_k u_k$. 
Using \eqref{eq:taylor}, we have
\begin{flalign*}
& G(x_k,\xi_k^+,\xi_k,u_k,\delta_k)\\
&=\frac{f(x+\delta_k u_k)-f(x_k)}{\delta_k}\frac{(d+1)u_k}{1+\lVert u_k \rVert^2}+\frac{(\eta_k^+-\eta_k)}{\delta_k}\frac{(d+1)u_k}{1+\lVert u_k \rVert^2},\\
&= \frac{(d+1)u_ku_k^T\nabla f(x_k)}{1+\lVert u_k \rVert^2} + \bigg(\frac{ u_k^T\nabla^2 f(\Bar{x}^+_k) u_k}{2}\bigg)\frac{(d+1)u_k}{1+\lVert u_k \rVert^2}\delta_k+\frac{(\eta_k^+-\eta_k)}{\delta_k}\frac{(d+1)u_k}{1+\lVert u_k \rVert^2},
% +\bigg(\frac{\delta_k \lVert u_k^T\nabla^2 f(\Bar{x}^+_k) u_k\rVert}{2}\bigg)\frac{(d+1)u_k}{1+\lVert u_k \rVert^2}\\
% &\qquad+\frac{(\eta_k^+-\eta_k)}{\delta_k}\frac{(d+1)u_k}{1+\lVert u_k \rVert^2},\\
% &\leq \frac{(d+1)u_ku_k^T\nabla f(x_k)}{1+\lVert u_k \rVert^2}+c_2'\delta_k\mathbf{1}_d\\
% &\qquad+\frac{(\eta_k^+-\eta_k)}{\delta_k}\frac{(d+1)u_k}{1+\lVert u_k \rVert^2},
\end{flalign*}
% where $w_k$ in $(a)$ is as defined in the lemma statement.
Let $u_k^i$ denote the $i$-th element of the random vector $u_k$.
Then, the off-diagonal elements of  $\displaystyle\E_u\bigg[\frac{(d+1)u_ku_k^T}{1+\lVert u_k \rVert^2}\bigg]$ satisfy
 $\displaystyle\E_u\bigg[\frac{(d+1)u_k^iu_k^j}{1+\lVert u_k \rVert^2}\bigg]= 0,$ since $\displaystyle\frac{1}{1+\lVert u_k\rVert^2}\in (0,1]$ is upper bounded by a fixed quantity, and for $i\neq j$, $\E_u[u_k^iu_k^j]=0$. Hence
\begin{equation}
\label{eq:4.6}
    \E_u\left[\frac{(d+1)u_ku_k^T}{1+\lVert u_k \rVert^2}\right]= c_2\mathbb{I},
\end{equation}
where $c_2$ is as defined in the lemma statement.

Notice that  $\E_u\left[(\eta_k^+-\eta_k)\frac{u_k}{1+\lVert u_k \rVert^2}\middle\vert\ \mathcal{F}_k\right]=0$, since $E[\eta_k^+-\eta_k\mid \mathcal{F}_k]=0$, and $u_k$ is independent of $\mathcal{F}_k$. Thus we obtain
\begin{align*}
    \E[G(x_k,\xi_k^+,\xi_k,u_k,\delta_k)|\mathcal{F}_k]=  c_2\nabla f(x_k)+\delta_k w_k.
\end{align*}
\end{proof}

\subsubsection{Proof of Theorem \ref{th1}}
\label{pf2}

For the proof of Theorem \ref{th1}, we require the notion of Lyapunov stability, which we define next.
\begin{definition}
\label{def:def2}
A continuously differentiable function $V\colon \mathbb{R}^d \to [0,\infty)$ is said to be a Lyapunov function for an ODE $\dot y=f(y)$ with set of equilibrium points $\mathit{H}$ if it satisfies the properties below.
\begin{enumerate}
    \item $\lim_{\lVert x \rVert \to +\infty} V(x) =\infty$.
    \item The inner product of $f(y)$ with $V(y)$ can take the following values:
    \[
    \langle f(y), \nabla V(y) \rangle=
    \begin{cases}
        0 & \text{if}\ y\in \mathit{H},\\
        < 0 & \text{otherwise}. \\
    \end{cases}
    \]
\end{enumerate}
\end{definition}

\begin{lemma}
\label{def:lem3}
Consider the function in \cref{def:eq1} such that $f(\cdot)\geq c'$ where $c'$ is a negative real number and $H$ be the set of equilibrium points of the ODE $\dot x(t)=-\nabla f(x(t))$, i.e., $H=\{x(t):\nabla f(x(t))=0\}$ where $f(x)$ is defined as in \cref{def:eq1}. Then $x(t) \to H$ as $t\rightarrow\infty$. 
\end{lemma}

\begin{proof}
Let $g(\cdot)=f(\cdot)-c'\geq 0$. Further
\begin{equation*}
    \begin{split}
        \frac{d g(x(t))}{d t}&= \langle \nabla f(x(t)), \dot{x}(t)\rangle\\
        &= \langle \nabla f(x(t)), -\nabla f(x(t))\rangle\\
        &=-\lVert \nabla f(x(t))\rVert^2.
    \end{split}
\end{equation*}
Thus, ${\displaystyle \frac{dg(x(t))}{dt} <0}$ for $x(t)\not\in H$ and is
0 otherwise. Thus, $g$ serves as a Lyapunov function for the above (gradient) ODE. The claim follows.
\end{proof}

\subsubsection*{Proof of Theorem \ref{th1}}
\begin{proof}
From \ref{def:eq2} and Lemma \ref{def:lem2}, we have
\begin{equation}
\label{eq:5}
    \begin{split}
       x_{k+1}&=x_{k}-\gamma_k(\E_u[G(x_k,\xi_k^+,\xi_k,u_k,\delta_k)|\mathcal{F}_k]+M_k),\\
       &=x_{k}-\gamma_k(c_2\nabla f(x_k)+ c_2'\delta_k \mathbf{1}_{d}+M_k).\\
    \end{split}
\end{equation}
where $M_k=G(x_k,\xi_k^+,\xi_k,u_k,\delta_k)-\E_u[G(x_k,\xi_k,\xi_k^+, u_k,\delta)|\mathcal{F}_k]$, $k\geq 0$, is a martingale difference sequence. Further, $\gamma_k, k\geq 1$ satisfies  \cref{def:ass3}.
The update rule \cref{eq:5} thus tracks the ODE \cref{ode}.  However the ODE \cref{ode} has the same equilibria as the ODE $$\dot{x}=-\nabla f(x).$$ In fact, if the constant $c_2>1$, the ODE will have a faster speed of convergence to $x^*$.  
Note that $c_2'\delta_k \mathbf{1}_{d}$ is equivalent to the bias vector. Here each term of $\mathcal{O}(\delta_k)\mathbf{1}_{d}$ goes to zero as $k \to \infty$. 
From \cite{Clark78} we can directly conclude the convergence of the above algorithm using the assumptions (a)-(d) below that are taken from \cite{Clark78}.
\begin{enumerate}[label=(\alph*)]
    \item $\nabla f(x)$ is a Lipschitz continuous function.
    \item The bias sequence $\delta_k,k\geq 1$ is almost surely bounded and with $\delta_k\to 0$ almost surely as $k\to \infty$. 
    % converges to zero for large $k$.
    \item The step-sizes $\gamma_k, k\geq 1$, satisfy $\gamma_k\to 0$ as $k\to\infty$ and $\sum_k \gamma_k=\infty$. 
    % and $\sum_k (\gamma_k)^2<\infty$.
    \item The Martingale sequence $M_k$ satisfies the following: $\forall \zeta >0$, 
    \begin{equation*}
        \lim_{k\to\infty}\p\left(\sup_{m\geq k}\lVert \sum_{i=k}^m \gamma_iM_i\rVert \geq \zeta\right)=0.
    \end{equation*}
\end{enumerate}
Assumptions $(a),(b)$ and $(c)$ above directly follow from \cref{def:ass3} and \cref{def:ass4}. 
We now verify Assumption $(d)$ above. 
Recall the Doob's martingale inequality, i.e.
\begin{equation*}
    \p\left(\sup_{m\geq 0}\lVert Z_m\rVert\geq \zeta\right)\leq \frac{1}{\zeta^2}\lim_{m\to \infty}\E\lVert Z_m \rVert^2.
\end{equation*}
By considering  $Z_m=\sum_{i=0}^{m-1} \gamma_iM_i$, we obtain
\begin{align}
\label{def:eq6}
    \p\left(\sup_{m\geq k}\lVert \sum_{i=k}^{m}\gamma_iM_i\rVert\geq \zeta\right)&\leq \frac{1}{\zeta^2}\lim_{m\to \infty}\E\lVert \sum_{i=k}^{m}\gamma_iM_i \rVert^2 \nonumber\\
        &\stackrel{(a)}{=}\frac{1}{\zeta^2}\sum_{i=k}^{\infty} \gamma_k^2\E\lVert M_i\rVert^2 ,
\end{align}
where the first inequality in $(\ref{def:eq6})$ follows from the fact that  $\E[M_iM_j]=\E[M_i\E[M_j]|\mathcal{F}_j]=0$. 
Now using the identity $\E\lVert X-\E[X|\mathcal{F}]\rVert^2\leq \E[\lVert X \rVert^2]$ for a r.v. $X$ one can rewrite $\E\lVert M_k\rVert^2\leq \E\lVert G(x_k,\xi_k^+,\xi_k,u_k,\delta_k)\rVert^2$. Hence
\begin{flalign}
\label{martingle}
&\E\lVert M_k \rVert^2 \nonumber \\ 
&\leq \E\lVert G(x_k,\xi_k^+,\xi_k,u_k,\delta_k)\rVert^2\nonumber\\
        &=\E\left[\frac{\lVert u_k\rVert^2(d+1)^2}{(1+\lVert u_k\rVert^2)^2}\left(\frac{\eta_k^+-\eta_k}{\delta_k}\right)^2\right]+\E\left[\frac{\lVert u_k\rVert^2(d+1)^2}{(1+\lVert u_k\rVert^2)^2}\left(\frac{f(x_k+\delta_ku_k)-f(x_k)}{\delta_k}\right)^2\right]\nonumber\\
        &\leq \frac{C''}{\delta_k^2}.
\end{flalign}
Here $C''=\E\left[\frac{\lVert u_k\rVert^2(d+1)^2}{(1+\lVert u_k\rVert^2)^2}\right]\big[2(\beta_1+\beta_2)\big] <\infty$.
The last inequality follows from $\cref{def:ass5}$. Moreover, from construction, the truncated Cauchy distribution  has finite moments. Plugging the above inequality in \cref{def:eq6},we obtain
\begin{equation*}
    \begin{split}
        \lim_{k\to\infty} \p\left(\sup_{m\geq k}\lVert \sum_{i=k}^{m}\gamma_iM_i\rVert\geq \zeta\right)\leq \frac{C''}{\zeta^2}\lim_{k\rightarrow\infty} \sum_{i=k}^\infty(\frac{\gamma_i^2}{\delta_i^2})=0
    \end{split}.
\end{equation*}
The last inequality follows by the assumption $\sum_k(\frac{\gamma_k^2}{\delta_k^2})< \infty$.
So, by the convergence of the martingale sequence and  Lemma \ref{def:lem3},  we can conclude that $x_k\to H$ almost surely as $k\to \infty$.
\end{proof}

\subsubsection{Proof of Theorem \ref{th2}}
\label{pf:th2}
\begin{proof}
\label{pf11}
This proof follows from Proposition 1 of \cite{chin1997comparative} after noting the following facts:
\begin{align}
\label{eq:4.10}
    \E\left[\frac{(u_k^i)^4}{(1+\lVert u \rVert^2)}\right]&\leq \E[\lVert u \rVert^4]\nonumber,\\
    \E\left[\frac{uu^T}{(1+\lVert u \rVert^2)}\right]&\preceq \mathbb{I},
\end{align}
% where $\Tilde{c}_2 = \frac{(u_k^i)^2}{1+\lVert u_k \rVert^2}$ and the inequality $\preceq$ is element-wise.
\end{proof}

\subsubsection{Proof of Proposition \ref{prop}}
\label{pf3}
For establishing the avoidance of traps result in Proposition \ref{prop}, we require a result from \cite{pemantle}, which we state below.
\begin{theorem}
\label{Them:1}
Let $Y$ be a function in $\Delta\subseteq \mathbb{R}^d$ such that $Y:\Delta\to T\Delta$ where $T\Delta$ is the tangent space of $\Delta$ at each point. Consider a sequence of random variables $\{x_n: n\geq 0\}$ that are updated as in \cref{def:eq4} with a given $x_0$. 
\begin{equation}
    \label{def:eq4}
    x_{k+1}=x_k+a_kY(x_k)+\mu_k.
\end{equation}
Let $p$ be any critical point, i.e., $Y(p)=0$, and let $\mathcal{N}$ be a neighborhood of $p$. Assume that there are constants $\gamma \in (1/2, 1]$ and $c_5,c_6,c_7,c_8\geq 0$ for which the following conditions are satisfied whenever $x_n\in \mathcal{N}$ and $n$ is sufficiently large:
\begin{itemize}
    \item[(i)] $p$ is a linearly unstable critical point of $Y$;
    \item[(ii)] $\frac{c_5}{k^\gamma} \leq a_k \leq \frac{c_6}{k^\gamma}$;
    \item[(iii)] $\E\left[(\mu_k \cdot \theta)^+|\mathcal{F}_k\right]\geq c_7/k^\gamma$ for every unit vector $\theta \in T\Delta$;
    \item[(iv)] $\lVert \mu_k \rVert\leq c_8/k^\gamma$.
\end{itemize}
In the above, $(\mu_k \cdot \theta)^+ \stackrel{\triangle}{=} \max \{\mu_k \cdot \theta,0\}$ is the positive part of $\mu_k \cdot \theta$.
Assume $Y$ is smooth enough (at least $C^2$) to apply the stable manifold theorem. Then $\p(x_k\to p) = 0.$
\end{theorem}

%%%%%%%%%%%%%%%%%%%%%%%%%%

\begin{remark}
Note that, the limit of $x_k$ exists if $Y(x)=0$. Here the iteration rule \cref{def:eq4} can be considered as a discrete version of the differential equation $\dot{x(t)}=Y(x(t))$ with initial condition $x_0=v(0)$. 
We need to classify the points where $Y(v)=0$. Consider a linear approximation $T$ near a critical point $p$ of $Y(\cdot)$ such that $Y(p+x)=T(x)+\mathcal{O}(|x|^2)$. $p$ is said to be attracting point if real part of the eigenvalue of T is negative and in such a case $x_t$ will converge to $p$ if there are no other attracting points for the ODE. On the other hand,  $p$ is said to be linearly unstable if some eigenvalue has positive real part and $x_n$ exists in the neighbourhood of $p$ for any choice of $v(0)$ which is not on a stable manifold of smaller dimension. However, if all the eigenvalues of $T$ have a positive real part, then $p$ is said to be a repelling node and the sequence $x_k, k\geq 0$ will never converge.
\cref{Them:1} gives conditions under which $\p(x_k\to p)=0$  when $p$ is a repelling point as well as a linearly unstable critical point.
\end{remark}
\subsubsection*{Proof of Proposition \ref{prop}}
\begin{proof}
Let's rewrite the update rule \cref{def:eq2} as
\begin{equation}
\label{eq-0}
x_{k+1} =x_{k}-\gamma_kG (x_k,u_k,\delta_k)-\epsilon_k,
\end{equation}
where
${\displaystyle G (x_k,u_k,\delta_k)=\frac{f(x_k+\delta_k u_k)-f(x_k)}{\delta_k}\frac{(d+1)u_k}{1+\lVert u_k \rVert^2}}$ and ${\displaystyle \epsilon_k=\gamma_k\frac{\eta_k^+-\eta_k}{\delta_k}\frac{(d+1)u_k}{1+\lVert u_k \rVert^2}}$.\\
Here $\E[\epsilon_k|\mathcal{F}_k]=0$. Note that \cref{eq-0} is equivalent to \cref{def:eq4}  by considering $-\epsilon_k=\mu_k$ and $-G =Y$. 

We will show that the conditions stated in \cref{Them:1} will hold for our case. Let $\gamma\in(\frac{1}{2},1]$. Choose the step-size $\gamma_k\in[\frac{c_8}{k^\gamma},\frac{c_9}{k^\gamma}]$ such that both $(ii)$ and \cref{def:ass3} are satisfied, and in \cref{def:ass5}, we have considered $\E|\eta_k^+|^2$ and $\E|\eta_k^+|^2$ are bounded which in turn implies that $| \eta_k^+-\eta_k|$ is bounded and noting the fact $u_k$ is the truncated Cauchy distribution (over the unit sphere), one can trivially show that $\lVert \epsilon_k\rVert$ is bounded, which implies $(iv)$. 
 
 We now show $(iii)$ for our case below.
Note that $\epsilon_k=(\eta_k^+-\eta_k)u_km_k$, where $m_k=\frac{(d+1)\gamma_k}{1+\lVert u_k \rVert^2}$.
Consider the unit vector with the $i$th entry as $1$, i.e., $\theta=(0,0,...1,..0)^T$. This implies $\epsilon_k\cdot\theta=(\eta_k^+-\eta_k)u_k^im_k$. Here, $u_k^i$ denotes the $i$th entry of the vector $u$ at the $k$th iteration.
\begin{align*}
    \E[(\epsilon_k\cdot\theta)^+]&=\E[((\eta_k^+-\eta_k)\cdot u_k^im_k)^+]\nonumber\\
        &\stackrel{(b)}{=}\E[\frac{(\eta_k^+-\eta_k)\cdot u_k^im_k+|(\eta_k^+-\eta_k)\cdot u_k^im_k|}{2}]\\
        &\stackrel{(c)}{=}\E[\frac{|(\eta_k^+-\eta_k)\cdot u_k^im_k|}{2}]\nonumber\\
        &=\E[\frac{|\eta_k^+-\eta_k|\cdot |u_k^im_k|}{2}]\nonumber\\
        &\stackrel{(d)}{\geq} \frac{c_9c_{10}}{2}.
\end{align*}
In the above, we used (i) the fact that $\max(x,y)=\frac{x+y+|x-y|}{2}$ to infer the equality in $(b)$; (ii) $\E[(\eta_k^+-\eta_k)\cdot u_k^im_k]=0$ since $E[\eta_k^+-\eta_k\mid \mathcal{F}_k]=0$ and $u_k$ is independent of $\mathcal{F}_k$, to infer the equality in $(c)$; and (iii) \cref{def:ass9} in conjunction with $\E_u|u_k^im_k|\geq c_{10}$ for some positive constant $c_{10}$ is used to infer $(c)$.
Thus, condition (iii) of \cref{Them:1} holds. The main claim now follows by an application of \cref{Them:1}.

\end{proof}

\subsection{Non-asymptotic convergence proofs}
\label{sec:appendix-nonasymp}
\subsubsection{Proof of Theorem \ref{coro2}}
\label{pf:coro2}
\begin{proof}
We use the the proof technique from \cite{Bhavsar2022}(in particular, proposition-1 there) in order to proof the main claim here. However unlike \cite{Bhavsar2022} we have a gradient estimate that comes from truncated Cauchy distribution.\\
Lets define $ \alpha_k \equiv (\xi_k,\xi_k^+,u_k,\delta), k\geq 1$ and $\alpha_{[N]}:=(\alpha_1,\alpha_2,...\alpha_N)$ \. Using Taylor series expansion over $f(x_k)$ for any $k=1,2,...,N$ the following is obtained
\begin{equation*}
    \begin{split}
        f(x_{k+1}) &\leq f(x_k)-\gamma_k\langle\nabla f(x_k),G(x_k,\alpha)\rangle+\frac{L}{2}\gamma_k^2 \left\lVert G(x_k,\alpha_k)\right\rVert^2,\\
        &=f(x_k)-c_2\gamma_k\left\lVert\nabla f(x_k)\right\rVert^2- \gamma_k\langle \nabla f(x_k) ,\Gamma_k\rangle
        +\frac{L}{2}\gamma_k^2 \left\lVert G(x_k,\alpha_k)\right\rVert^2.
    \end{split}
\end{equation*}
Here $\Gamma_k \equiv G(x_k,\xi_k,u_k,\delta)-c_2\nabla f (x) \equiv G (x_k,\alpha_k) - c_2\nabla f(x_k)$.
Adding upto $N$-terms both side of these inequalities and considering $f^*\leq f(x_{N+1})$ and $\gamma_k=\gamma$ for all $k$,
we obtain
\begin{align}
    \label{def:eq12}
        \sum_{k=1}^N c_2\gamma \left\lVert\nabla f(x_k)\right\rVert^2
        & \leq f(x_1)-f^*-\sum_{k=1}^N \gamma \langle\nabla f(x_k),\gamma\rangle +\frac{L}{2}\sum_{k=1}^N \gamma ^2\left\lVert G(x_k,\alpha_k)\right\rVert^2.
\end{align}
% $-\lVert V \rVert_1\leq \sum_{k=1}^{d} v_k$ for a $d-$dimensional vector $V$
Now by Lemma \ref{def:lem2} we have
\[\E_{\alpha_{[k]}}[\Gamma_k] =\E_{\alpha_{[k]}}[\Gamma_k|\alpha_{k-1}]=\E_{\alpha_{[k]}}[\Gamma_k|x_k]=\E_{\alpha[k]}[G-c_2\nabla f|x_k]\stackrel{(e)}{\leq} \tau\mathbf{1}_{d\times1},\] 
notice $\tau=c_2'\delta$ where $c_2'=(d+1)B$ is a constant that arise from the Taylor series as defined. In the above vector inequality $(e)$ is element-wise. From \cref{martingle} we have
\[\E_{\alpha_{[k]}}[\lVert G \rVert^2]\leq \lVert\E_{\alpha_{[k]}}[ G]\rVert^2+\frac{C''}{\delta^2}.\]
Hence by taking the expectation w.r.t $\alpha_{[N]}$ on both side of \cref{def:eq12} the following is obtained
\begin{flalign*}
       \sum_{k=1}^N c_2\gamma \E_{\alpha_{[N]}}\left\lVert\nabla f(x_k)\right\rVert^2
        & \leq D+BNc_2'\delta \gamma +\frac{L}{2}\sum_{k=1}^N \gamma^2\Big[\E_{\alpha_{[N]}}\lVert \nabla f(x_k)\lVert^2+2c_2'\delta B+dc_2'^2\delta^2+\frac{C''}{\delta^2}\Big] .
\end{flalign*}
The above inequality uses the fact $-\lVert V \rVert_1\leq \sum_{k=1}^{d} v_k$ for a $d-$dimensional vector $V$ followed by $\lVert \nabla f(x_k)\rVert_1\leq \lVert \nabla f(x_k)\rVert\leq B$ from \cref{def:ass1}. Note that $D=f(x_1)-f^*$.
By rearranging the terms we have
\begin{flalign*}
          \left[c_2\gamma-\frac{L\gamma^2}{2}\right]\sum_{k=1}^N\E_{\alpha_{[N]}}\left\lVert\nabla f(x_k)\right\rVert^2
        &\leq D +BNc_2'\delta \gamma +\frac{LN}{2}\left(2c_2'\delta B+dc_2'^2\delta^2+\frac{C''}{\delta^2}\right) \gamma^2.
\end{flalign*}
Due to the choice of $\gamma_k = \gamma=\min\left\{\frac{c_2}{L},\frac{1}{N^{2/3}}\right\}$, it is obvious that $N\left[c_2\gamma-\frac{L\gamma^2}{2}\right]\geq0$. Thus by dividing both sides of the above inequality by $N\left[c_2\gamma-\frac{L\gamma^2}{2}\right]$ and noting the fact
\begin{equation}
\label{eq98}
    \p_R(k)=Prob(R=k)=\frac{\left[c_2\gamma-\frac{L\gamma^2}{2}\right]}{N\left[c_2\gamma-\frac{L\gamma^2}{2}\right]}=\frac{1}{N},
\end{equation}
the following is obtained
\begin{flalign*}
         \E_{\alpha_{[N]}}\left\lVert\nabla f(x_R)\right\rVert^2
        &\leq\frac{1}{N\left[c_2\gamma-\frac{L\gamma^2}{2}\right]}\bigg[D +BNc_2'\delta \gamma +\frac{LN}{2}\left(2c_2'\delta B+dc_2'^2\delta^2+\frac{C''}{\delta^2}\right) \gamma^2\bigg].
\end{flalign*}
By considering, $\gamma=\frac{c_2}{L}$, the following is obtained 
% Note that, the condition of $\gamma_k,\delta$  is given by
% \[
%     \gamma_k = \gamma=\min\left\{\frac{c_2}{L},\frac{1}{N^{2/3}}\right\},\delta=\dfrac{1}{N^{1/6}} \quad k=1,2,...N,
% \]Thus one can have,
\begin{equation*}
   Nc_2\gamma\left[1 -\frac{L}{2c_2}\gamma\right]\geq \frac{Nc_2\gamma}{2}.
\end{equation*}
From the above inequality, we can write
\begin{flalign*}
\E\left[\left\lVert\nabla f(x_R)\right\rVert^2\right]
        & \leq \frac{2D} {Nc_2\gamma}+\frac{2Bc_2'\delta}{c_2}+\frac{L}{c_2}\left(2c_2'\delta B+dc_2'^2\delta^2+\frac{C''}{\delta^2}\right)\gamma \\
         &\leq \frac{2D}{Nc_2}\max\left\{\frac{L}{c_2}, N^{2/3}\right\}+\frac{2Bc_2'\delta}{c_2}+\frac{L}{c_2N^{2/3}}\left(2c_2'\delta B+dc_2'^2\delta^2+\frac{C''}{\delta^2}\right)\\
         &\stackrel{(g)}{\leq} \Big(\frac{2DL}{Nc_2^2}+ \frac{2D}{N^{1/3}}\Big)+\frac{2Bc_2'}{c_2N^{1/6}}+\frac{L}{c_2N^{2/3}}\left(\frac{2Bc_2'}{N^{1/6}} +\frac{dc_2'^2}{N^{1/3}}+\frac{C''}{N^{-1/3}}\right)\\
         &= \Big(\frac{2DL}{c_2N}+ \frac{2D}{N^{1/3}}\Big)+\frac{2Bc_2'}{c_2N^{1/6}}+\frac{L}{c_2}\left(\frac{2Bc_2'}{N^{5/6}} +\frac{dc_2'^2}{N}+\frac{C''}{N^{1/3}}\right).\\
\end{flalign*}
Note that $(g)$ uses the condition of $\delta=\frac{1}{N^{1/6}}$. Hence proved.
\end{proof}

\subsubsection{Proof of Theorem \ref{coro1}}
The proof is accumulated from a sequence of lemmas..We follow the technique from \cite{gadimi}, and our proof of the lemmas involves significant deviations owing to the fact that biased gradient information under truncated Cauchy distribution are available instead of unbiased gradient information (under Gaussian distribution). 
\begin{lemma}
\label{def:lem6}
let $f$ be the function satisfying \cref{def:ass1} and \cref{def:ass2} and $\delta$ be the smoothing parameter defined in \cref{alg:alg1}. Then
\begin{equation*}
     \lVert f_\delta(x) \rVert \leq \frac{c_{11}}{d+1}\left [\frac{L\delta^2}{2}+2\delta B\right].
\end{equation*}
\end{lemma}

\begin{proof}
\begin{equation*}
\begin{split}
     f_\delta(x) &=\E_u(f(x+\delta u)-f(x)) \\
     &=\E_u[f(x+\delta u)-f(x)-\delta\langle\nabla f(x),u\rangle]+\delta\E_u[\langle\nabla f(x),u \rangle] \\
     |f_\delta(x)| 
    %  &\leq \frac{L\delta^2}{2}\E[\lVert u \rVert^2]+\delta\lVert \nabla f(x)\rVert [\lVert u \rVert]\\
     &\stackrel{(d)}{\leq}\frac{L\delta^2}{2}\E_u[\lVert u \rVert^2]+\delta\lVert\nabla f(x) \rVert \E_u[\lVert u \rVert]\\
    %  &\leq \frac{c_1L\delta^2}{2(d+1)}+\delta B \frac{c_1}{d+1/2}\\
     &\stackrel{(d)}{\leq}\frac{\delta^2c_{11}L}{2(d+1)}+ \frac{2\delta Bc_{11}}{2d+1}\\
     &\leq \frac{c_{11}}{d+1}\left[\frac{L\delta^2}{2}+2\delta B\right].
\end{split}
\end{equation*}
where $(d)$ follows from smoothness of $f$ and Cauchy-Schwarz inequality and $(e)$ follows from Lemma \ref{def:lem5}. 
\end{proof}

\begin{proposition}
\label{prop1}
The solution of a linear system of equations $Ax=y$, where $A$ is a non-invertible matrix, is $Py$ where $P$ is the generalized inverse of the matrix $A$ that satisfies $APA=A$ and $PAP=P$.
\end{proposition}

\begin{lemma}
\label{def:lem7}
Under \cref{def:ass1} and \cref{def:ass2}, we have
\begin{enumerate}[label=(\alph*)]
\item 
\begin{equation*}
    \lVert \nabla f(x) \rVert^2 \leq 2c_{12} \lVert\nabla f_\delta(x) \rVert^2+\frac{c_{11}c_{12}\delta^2L^2(d+1)}{2}.
\end{equation*}
\item 
\begin{equation*}
    \lVert \nabla f_\delta(x) \rVert^2 \leq \frac{2c_{11}}{d+1}\left[\frac{(d+1)^2\delta^2L^2}{4}+\lVert\nabla f(x) \rVert^2\right].
\end{equation*}
\end{enumerate}
\end{lemma}
\begin{proof}
Notice that
\begin{equation}
\label{def:eq11}
    \begin{split}
        \E_u \left[\frac{(d+1)u}{1+\lVert u \rVert^2}\langle \nabla f(x),u\rangle\right]
        &= A\nabla f(x)\\
        \end{split}
\end{equation}
where $A=\E_u\left(\frac{(d+1) uu^T}{1+\lVert u \rVert^2}\right)$ is a matrix. Let $P$ be the generalized inverse of $A$ which satisfies the condition in \cref{prop1}. Now

\begin{align*}
        \nabla f(x) &\stackrel{(e)}{=}P\cdot\E\bigg[\frac{(d+1)u}{1+\lVert u \rVert^2}\langle \nabla f(x),u\rangle\bigg],\\
        &= \frac{P}{\delta} \E\left[\frac{(d+1)u}{1+\lVert u \rVert^2}\{f(x+\delta u)-f(x)\}\right]
        -\frac{P}{\delta}\E\bigg[\frac{(d+1)u}{1+\lVert u \rVert^2}\{f(x+\delta u)-f(x)
         -\delta\langle \nabla f(x),u\rangle\}\bigg],\\
\end{align*}
$(e)$ is a simple application of Proposition \ref{prop1}. Let $\lVert\cdot\rVert_F$ denotes the Frobenius norm of the matrix. Then $\lVert Px \rVert_2\leq c_{12}\lVert x \rVert_2$ where $c_{12}=\lVert P\rVert_F$
\begin{align*}
        \lVert\nabla f(x)\rVert^2 &\stackrel{(f)}{\leq} 2\lVert P \rVert_F^2 \lVert \nabla f_\delta(x) \rVert^2+\frac{2\lVert P \rVert_F^2}{\delta^2}\E\left[\frac{(d+1)^2\lVert u \rVert^6 L^2 \delta^4}{4(1+\lVert u \rVert^2)^2}\right],\\
        &=2c_{12} \lVert\nabla f_\delta(x) \rVert^2+
        \frac{c_{12}\delta^2L^2(d+1)^2}{2}\E\left[\frac{\lVert u \rVert^6}{(1+\lVert u \rVert^2)^2}\right],\\
        &\leq 2c_{12} \lVert\nabla f_\delta(x) \rVert^2+\frac{c_{12}\delta^2L^2(d+1)^2}{2}\E\left[\lVert u \rVert^6\right],\\
        &\leq 2c_{12} \lVert\nabla f_\delta(x) \rVert^2+\frac{c_{12}\delta^2L^2(d+1)^2}{2}\E\left[\lVert u \rVert^6\right]\\
        &\leq 2c_{12} \lVert\nabla f_\delta(x) \rVert^2+\frac{c_{11}c_{12}\delta^2L^2(d+1)^2}{2(d+3)},\\
        &\leq 2c_{12} \lVert\nabla f_\delta(x) \rVert^2+c_{11}c_{12}\delta^2L^2(d+1),\\
\end{align*}
where $(f)$ uses the Cauchy-Schwarz inequality followed by Lemma \ref{def:lem5}.

We now prove the second claim.
\begin{equation*}
   \begin{split}
        \left[f(x+\delta u)-f(x)\right]^2 &=[f(x+\delta u)-f(x) -\delta\langle\nabla f(x),u\rangle+\delta\langle\nabla f(x),u\rangle]^2,\\
        & \leq 2\left(\frac{\delta^2}{2} L \left\lVert u \right\lVert^2\right)^2 + 2\delta^2\langle \nabla f(x),u\rangle^2,\\
        & \leq 2\left(\frac{\delta^2}{2} L \left\lVert u \right\lVert^2\right)^2 + 2\delta^2 \lVert\nabla f(x)\rVert^2\left\lVert u\right\rVert^2.\\
    \end{split} 
\end{equation*}
Now, by taking the norm in the both side of the inequality and employing \cref{def:lem5} the following is obtained 
\begin{flalign*}
        &\lVert \nabla f_\delta(x) \rVert^2
        \leq \frac{1}{\delta^2}\E_u \left[\left(f(x+\delta u)-f(x)\right)^2 \frac{(d+1)^2\left\lVert u \right\lVert^2}{(1+\lVert u \rVert^2)^2}\right],\\
        &\leq \frac{((d+1)\delta L)^2}{2} \E_u\left[\frac{\left\lVert u \right\lVert^6}{(1+\lVert u \rVert^2)^2}\right]
        +2\lVert\nabla f(x)\rVert^2(d+1)^2\E_u \left[\frac{\left\lVert u \right\lVert^4}{(1+\lVert u \rVert^2)^2}\right],\\
        &\leq \frac{((d+1)\delta L)^2}{2} \E_u\left\lVert u \right\lVert^6+2\lVert\nabla f(x)\rVert^2(d+1)^2\E_u \left\lVert u \right\lVert^4,\\
        &\leq \frac{((d+1)\delta L)^2c_{11}}{2(d+3)}+\frac{2\lVert\nabla f(x)\rVert^2c_{11}}{d+2},\\
        &\leq \frac{2c_{11}}{d+1}\left[\frac{(d+1)^2\delta^2L^2}{4}+\lVert\nabla f(x) \rVert^2\right].
\end{flalign*}
\end{proof}

In the following lemma we state and prove a general result that holds true for any choice of non-increasing stepsize sequence, smoothing parameter. Subsequently, we specialize the result for the choice of parameters suggested in \cref{coro1}, to prove the same.

\begin{lemma}
\label{Them:2}
Suppose, $x_k$ generated by \cref{alg:alg1} and $\{\gamma_k\}$ be the desired step-sizes for the iteration. Let and the probability mass function $P_R(\cdot)$  are chosen such that 
$\gamma_k\leq\frac{1}{2Lc_{13}}$ and 
\begin{equation}
\label{prob1}
    P_R(k):=Prob(R=k)=\frac{\left[\gamma_k -Lc_{13}\gamma_k^2\right]}{\sum_{k=1}^N\left[\gamma_k -Lc_{13}\gamma_k^2\right]},
\end{equation}
where $c_{13} = \frac{4c_{11}c_{12}}{d+1}.$ and $c_{11}$,$c_{12}$ are defined in \cref{def:eq10} ,\cref{def:eq11} respectively. Then under \cref{def:ass1},\cref{def:ass10} and for any $N\geq 1$
\begin{flalign}
\label{eq:lem6}
\E\left[\left\lVert\nabla f(x_R)\right\rVert^2\right]
&\leq \frac{1}{\sum_{k=1}^N\left[\gamma_k -Lc_{13}\gamma_k^2\right]}\bigg[\frac{L\delta^2c_{13}}{2}+2c_{13}\delta B +L^2\delta^2d^2c_{13}\sum_k\gamma_k\\
        &+Lc_{13}\left(\frac{d\delta^2L^2}{4}+ \sigma^2\right)\sum \gamma_k^2\bigg].
\end{flalign}
\end{lemma}
\label{pf:Them2}
\begin{proof}
Let $ \alpha_k \equiv (\xi_k,u_k), k\geq 1$ and $\alpha_{[N]}:=(\alpha_1,\alpha_2,...\alpha_N)$. \ 
Denote $\Gamma_k \equiv G_\delta(x_k,\xi_k,u_k)-\nabla f_\delta (x) \equiv G_\delta (x_k,\alpha_k) - \nabla f_\delta(x)$. Now, by Taylor series expansion $f_\delta(x_{k+1})$ we have for any $k=1,2,...,N$,

\begin{equation*}
    \begin{split}
        f_\delta(x_{k+1}) &\leq f_\delta(x_k)-\gamma_k\langle\nabla f_\delta(x_k),G_\delta(x_k,\alpha)\rangle+\frac{L}{2}\gamma_k^2 \left\lVert G_\delta(x_k,\alpha_k)\right\rVert^2\\
        &=f_\delta(x_k)-\gamma_k\left\lVert\nabla f_\delta(x_k)\right\rVert^2- \gamma_k\langle \nabla f_\delta(x_k) ,\Gamma_k\rangle
        +\frac{L}{2}\gamma_k^2 \left\lVert G_\delta(x_k,\alpha_k)\right\rVert^2.
    \end{split}
\end{equation*}
Adding upto $N$-terms both side of these inequalities and applying $f_\delta^*\leq f_\delta(x_{N+1})$,
we obtain
\begin{align}
    \label{def:eq5}
        \sum_{k=1}^N \gamma_k \left\lVert\nabla f_\delta(x_k)\right\rVert^2
        & \leq f_\delta(x_1)-f_\delta^*-\sum_{k=1}^N \gamma_k \langle\nabla f_\delta(x_k),\Gamma_k\rangle +\frac{L}{2}\sum_{k=1}^N \gamma_k ^2\left\lVert G_\delta(x_k,\alpha_k)\right\rVert^2.
\end{align}
From the unbiased property of $G_\delta(x_k,u_k,\xi_k)$ the following holds
\begin{equation*}
    \E[\langle\nabla f_\delta(x_k),\Gamma_k\rangle|\alpha_{[k-1]}]=0.
\end{equation*}
% Now by the assumption $F(\cdot,\xi_k)\in \mathcal{C}_L^{[1,1]}$ and \cref{def:lem7}(b) ,
Now by \cref{def:ass10} and Lemma \ref{def:lem7}(b)
\begin{equation}
\label{eq}
\begin{split}
    \E[\left\lVert G_\delta(x_k,\alpha_k) \right\rVert^2|\alpha_{[k-1]}] 
    &\leq \frac{2c_{11}}{d+1}\left[\E\lVert \nabla F(x_k,\xi_k) \rVert^2|\alpha_{[k-1]}]+\frac{(d+1)^2\delta^2L^2}{4}\right]\\
     &\leq  \frac{2c_{11}}{d+1}\left[2\E[\left\lVert \nabla f(x_k) \right\rVert^2|\alpha_{[k-1]}+\sigma^2]+\frac{(d+1)^2\delta^2L^2}{4}\right].
\end{split}
\end{equation}

Note that,the second inequality implies from variance bound of \cref{def:ass1}. Taking expectations with respect to $\alpha_{[N]}$ on both sides of \cref{def:eq5}, we have
\begin{flalign*}
&\sum_{k=1}^N \gamma_k \E_{\alpha_{[N]}}\left\lVert\nabla f_\delta(x_k)\right\rVert^2 \\
&\leq f_\delta(x_1)-f_\delta^*+
        \frac{L}{2}\sum_{k=1}^N \gamma_k^2\bigg[\frac{2c_{11}}{d+1}\Big[2\E[\left\lVert \nabla f(x_k) \right\rVert^2|\alpha_{[k-1]}+\sigma^2]
        +\frac{(d+1)^2\delta^2L^2}{4}\Big]\bigg].\\
\end{flalign*}
Applying Lemma \ref{def:lem7}(a) in the r.h.s and rearranging the term we obtain
\begin{flalign*}
        &\sum_{k=1}^N \gamma_k \left[\frac{1}{2c_{12}}\E_{\alpha_{[N]}}\left\lVert\nabla f(x_k)\right\rVert^2-\frac{c_{11}(d+1)L^2\delta^2}{2}\right]\\
        &\leq f_\delta(x_1)-f_\delta^* +\frac{LNc_{11}}{d+1}\left(\frac{(d+1)\delta^2L^2}{4}+ 2\sigma^2\right) \sum_{k=1}^N\gamma_k^2+\frac{2Lc_{11}}{d+1}\sum_{k=1}^N \gamma_k^2\E[\lVert\nabla f(x_k) \rVert^2].
\end{flalign*}
From \cref{def:lem6} and using the fact  $f_\delta(x_1)-f_\delta^*\leq |f_\delta(x_1)-f_\delta^*|\leq |f_\delta(x_1)|+|f_\delta^*|$ one can have
\begin{flalign*}
        &\sum_{k=1}^N \left[\frac{\gamma_k}{2c_{12}}-\frac{2Lc_{11}\gamma_k^2}{d+1}\right] \E_{\alpha_{[N]}}\left\lVert\nabla f(x_k)\right\rVert^2\\ &\leq \frac{c_{11}L\delta^2}{d+1}+\frac{4c_{11}\delta B}{d+1} +\frac{NL^2\delta^2(d+1)c_{11}}{2}\sum_{k=1}^N\gamma_k+\frac{LNc_{11}}{d+1}\left(\frac{(d+1)\delta^2L^2}{4}+ 2\sigma^2\right) \sum_{k=1}^N\gamma_k^2.
\end{flalign*}
Thus by rearranging the above inequality we obtain

\begin{flalign*}
        &\sum_{k=1}^N \left[\gamma_k-\frac{4c_{11}c_{12}L\gamma_k^2}{d+1}\right] \E_{\alpha_{[N]}}\left\lVert\nabla f(x_k)\right\rVert^2\\ &\leq \frac{2c_{11}c_{12}L\delta^2}{d+1}+\frac{8c_{11}c_{12}\delta B}{d+1} +L^2\delta^2(d+1)c_{11}c_{12}\sum_{k=1}^N\gamma_k+\frac{2Lc_{11}c_{12}}{d+1}\left(\frac{(d+1)\delta^2L^2}{4}+ 2\sigma^2\right)\sum_{k=1}^N \gamma_k^2.
\end{flalign*}
Let $c_{13} = \frac{4c_{11}c_{12}}{d+1}$
\begin{flalign}
\label{eq:lem6_1}
        &\sum_{k=1}^N \left[\gamma_k-Lc_{13}\gamma_k^2\right] \E_{\alpha_{[N]}}\left\lVert\nabla f(x_k)\right\rVert^2\\ &\leq \frac{L\delta^2c_{13}}{2}+2c_{13}\delta B +\frac{L^2\delta^2(d+1)^2c_{13}}{4}\sum_{k=1}^N\gamma_k+\frac{Lc_{13}}{2}\left(\frac{(d+1)\delta^2L^2}{4}+ \sigma^2\right)\sum_{k=1}^N \gamma_k^2.
\end{flalign}
Dividing both sides of the above inequality by $\sum_{k=1}^N\left[\gamma_k -Lc_{13}\gamma_k^2\right]$ and notice that
\begin{equation}
\label{eq100}
    \begin{split}
        \E\left[\left\lVert\nabla f(x_R)\right\rVert^2\right]&=\E_{R,\alpha_{[N]}}\left[\left\lVert\nabla f(x_R)\right\rVert^2\right] 
        =\frac{\sum_{k=1}^N \left[\gamma_k -Lc_{13}\gamma_k^2\right]\E_{\alpha_{[N]}}\left\lVert\nabla f(x_k)\right\rVert^2}{\sum_{k=1}^N\left[\gamma_k -Lc_{13}\gamma_k^2\right]},
    \end{split}
\end{equation}
we obtain \cref{eq:lem6} by replacing $d+1$ with $2d$ in \cref{eq:lem6_1}.

\end{proof}
We now specialize the result obtained from \cref{Them:2} to get the tight bound for \cref{alg:alg1} as describe in \cref{coro1}.
\subsubsection*{Proof of Theorem \ref{coro1}}
\begin{proof}
\label{pf:coro1}
Recall the step-size $\gamma_k=\gamma$, smoothing parameter $\delta_k=\delta,\forall k\geq 1,$ where
\begin{equation}
\label{eq-97}
    \begin{split}
        \gamma_k=\gamma &=\min\left\{\frac{1}{2Lc_{13}},\frac{1}{c_{13}\sigma\sqrt{N}}\right\}\ , \delta = \frac{1}{L\sqrt{dNc_{13}}}.
    \end{split}
\end{equation}
From the above condition of $\gamma$ by considering $\gamma\leq\frac{1}{2Lc_{13}}$ we have
% \[
%     \gamma_k\leq \frac{1}{2Lc_{13}} \quad k=1,2,...N
% \]
\begin{equation}
    N\gamma\left[1 -Lc_{13}\gamma\right]\geq \frac{N\gamma}{2}.
\end{equation}
Henceforth, from the above inequality and by \cref{Them:2}, we obtain
\begin{flalign*}
       &\E\left[\left\lVert\nabla f(x_R)\right\rVert^2\right]\\
        & \stackrel{(h)}{\leq} \frac{L\delta^2c_{13}}{N\gamma}+\frac{4\delta Bc_{13}}{N\gamma}+2L^2\delta^2d^2c_{13}+2Lc_{13}\left(\frac{d\delta^2L^2}{4}+\sigma^2\right)\gamma, \\
        & \leq \left(\frac{L\delta^2c_{13}}{N}+\frac{4\delta Bc_{13}}{N}\right) \max\{2Lc_{13},\sigma c_{13} \sqrt{N}\}+2L^2\delta^2d^2c_{13}+ \frac{d\delta^2L^2}{4}+\frac{2\sigma L}{\sqrt{N}},\\
        &\leq \left(\frac{L\delta^2c_{13}}{N}+\frac{4\delta Bc_{13}}{N}\right)(2Lc_{13}+\sigma c_{13}\sqrt{N})+L^2\delta^2d(2c_{13}d+1)+\frac{4\sigma L}{\sqrt{N}},\\
        &\stackrel{(k)}{=} \frac{c_{13}^2}{N}(L\delta^2+4\delta B)(2L+\sigma \sqrt{N})+L^2\delta^2d(2c_{13}d+1)+\frac{2\sigma L}{\sqrt{N}},\\
        &\leq \frac{c_{13}}{LN}(\frac{1}{Nd}+\frac{4 B\sqrt{c_{13}}}{\sqrt{Nd}})(2L+\sigma \sqrt{N})+\frac{(2c_{13}d+1)}{Nc_{13}}+\frac{2\sigma L}{\sqrt{N}},\\
        &=\frac{c_{13}}{LN}(\frac{2L}{Nd}+\frac{\sigma}{d\sqrt{N}}+\frac{8LB\sqrt{c_{13}}}{\sqrt{dN}}+\frac{\sigma B\sqrt{c_{13}}}{\sqrt{d}})+\frac{(2c_{13}d+1)}{Nc_{13}}+\frac{2\sigma L}{\sqrt{N}},\\
        &\leq \frac{c_{13}}{LN}(2L+\sigma+8LB\sqrt{c_{13}}+\sigma B\sqrt{c_{13}})+\frac{(2c_{13}d+1)}{Nc_{13}}+\frac{2\sigma L}{\sqrt{N}},\\
        &=\frac{c_{14}}{N}+\frac{4\sigma L}{\sqrt{N}}.
\end{flalign*}
Where $c_{14}=2L+c_{13}\sigma+8LBc_{13}^{3/2}+\frac{\sigma B c_{13}^{3/2}B}{L}+\frac{2c_{13}d+1}{L}$. Notice that, $(h)$ uses the condition of $\gamma_k$ in \cref{eq-97}  and $(k)$ follows from the condition of $\delta=\frac{1}{L\sqrt{dNc_{13}}}$. Thus by rearranging the terms as noting the fact that $1/Nd \leq 1$ we get the convergence rate $\mathcal{O}(1/\sqrt{N}).$
\end{proof}

\subsection{Non-asymptotic analysis for the balanced estimator}
\label{appendix-D}
Till now we have covered the analysis with imbalance gradient estimator. However, with this estimator we will get high bias which is $\mathcal{O}(\delta_k)$. We will now introduce a balanced gradient estimator and later we will show how it will help to achieve low bias as compared to the previous one.\\ 
Recall the finite difference gradient estimate is defined as
\begin{equation}
\label{eq-1}
    \begin{split}
    \nabla f_\delta(x) =\frac{1}{\delta} \E_{h(u)} \left[(f(x+\delta u)-f(x)) 
        \frac{(d+1)u}{(1+\lVert u \rVert^2)}\right].\\
    \end{split}
\end{equation}
Note that by change of variable we can rewrite \cref{eq-1} as
\begin{equation}
\label{eq-2}
    \begin{split}
    \nabla f_\delta(x) =\frac{1}{\delta} \E_{h(u)} \left[(f(x)-f(x-\delta u)) 
        \frac{(d+1)u}{(1+\lVert u \rVert^2)}\right].\\
    \end{split}
\end{equation}
Now by summing up \cref{eq-1} and \cref{eq-2} one can get the balanced estimator as
\begin{equation}
\label{eq-3}
    \begin{split}
    \nabla f_\delta(x) =\frac{1}{2\delta} \E_{h(u)} \left[(f(x+\delta u)-f(x-\delta u)) 
        \frac{(d+1)u}{(1+\lVert u \rVert^2)}\right].\\
    \end{split}
\end{equation}
Thus for the case of noisy function measurements the balanced estimator is
\begin{flalign}
\label{eq-4}
    \tilde G (x_k,\xi_k^+,\xi_k^-,u_k,\delta_k)
&\stackrel{\triangle}{=} {\displaystyle 
\bigg(\frac{F(x_k+\delta_k u_k,\xi_k^+)-F(x_k-\delta_ku_k,\xi_k^-)}{2\delta_k}\bigg)\frac{(d+1)u_k}{1+\lVert u_k \rVert^2}}.
\end{flalign}

\subsubsection{Proof of Theorem \ref{prop4}}
\label{pf:th4}
\begin{lemma}
\label{prop3}
Under \cref{def:ass3}-\cref{def:ass6}, and $\tilde G$ defined in \cref{balanced} 
we have almost surely
\begin{align}
	\label{eq:grad-bias_1}
        \E[\tilde G (x_k,\xi_k,\xi_k^+, u_k,\delta_k)|\mathcal{F}_k]=c_2\nabla f(x_k)+ c_2''\delta^2_k \mathbf{1}_{d},
\end{align}
where $c_2=\E_{h(u)}\left[\frac{(d+1)(u_k^i)^2}{1+\lVert u_k \rVert^2}\right]$, $c_2''=\frac{B_1d^4}{3}$ and $\mathbf{1}_{d}$ is the $d$-dimensional vector of all ones. 
\end{lemma}

\begin{proof}
Using Taylor series expansion for truncated Cauchy perturbations we obtain: $f(x_k+\delta_k u_k)-f(x_k-\delta_k u_k)$ $=$ $2\delta_k u_k^T\nabla f(x_k)+ \frac{\delta_k^3}{6}(\nabla^3f(\Bar{x_k}^+)+\nabla^3f(\Bar{x_k}^-))(u_k\otimes u_k \otimes u_k)$.\\ 
Here $\otimes$ denotes the Kronecker product and $\Bar{x}^+$(respectively,
$\Bar{x}^-$) are on the line segment between $x$ and $x+\delta u$ (respectively,
$x-\delta u)$. So
\begin{flalign*}
\E[\tilde G(x_k,\xi_k,\xi_k^+,u_k,\delta_k)|\mathcal{F}_k]
&=\E\Big[\frac{F(x+\delta_k u_k,\epsilon_k^+)-F(x-\delta_k u_k,\epsilon_k^-)}{2\delta_k}\frac{(d+1)u_k}{1+\lVert u_k \rVert^2}\vert \mathcal{F}_k\Big]\\
&=\E\Big[\frac{f(x+\delta_k u_k)-f(x-\delta_k u_k)}{2\delta_k}\frac{(d+1)u_k}{1+\lVert u_k \rVert^2}\vert \mathcal{F}_k\Big]+\E\Big[\frac{\eta_k^+-\eta_k^-}{2\delta_k}\frac{(d+1)u_k}{1+\lVert u_k \rVert^2}\vert \mathcal{F}_k\Big],\\
&=\E\Big[\frac{(d+1)u_ku_k^T\nabla f(x_k)}{1+\lVert u_k \rVert^2}\vert \mathcal{F}_k\Big]+\E\Big[\frac{\delta_k^2(d+1)u_k}{12(1+\lVert u_k \rVert^2)}(\nabla^3f(\Bar{x_k}^+)+\nabla^3f(\Bar{x_k}^-))(u_k\otimes u_k \otimes u_k)\vert \mathcal{F}_k\Big],\\
&\leq c_2\nabla f(x_k)+\E\Big[\frac{\delta_k^2 (d+1)u_k}{12}(\nabla^3f(\Bar{x_k}^+)+\nabla^3f(\Bar{x_k}^-))(u_k\otimes u_k \otimes u_k)\vert \mathcal{F}_k\Big].
\end{flalign*}
Now the $j$th coordinate of the second term in RHS of the above inequality is bounded as follows:
\begin{flalign*}
\E\Big[\frac{\delta_k^2u_k^j(d+1)}{12}(\nabla^3f(\Bar{x_k}^+)+\nabla^3f(\Bar{x_k}^-))(u_k\otimes u_k \otimes u_k)|\mathcal{F}_k\Big]
&\leq \frac{B_1\delta_k^2(d+1)}{6}\sum_{l_1=1}^{d}\sum_{l_2=1}^{d}\sum_{l_3=1}^{d}\E(u_k^j u_k^{l_1}u_k^{l_2}u_k^{l_3})
\leq \frac{B_1 d^4\delta_k^2}{3}.
\end{flalign*}
The first inequality follows from \cref{def:ass4} and in the last one we use the fact $|u_k^l|\leq1$.
% of r The rest proof follows from \cref{pf1}
\end{proof}

\subsubsection*{Proof of \cref{prop4}}
\begin{proof}
Notice by \cref{prop3}, \[\E_{\alpha_{[k]}}[\Gamma'_k] =\E_{\alpha_{[k]}}[\Gamma'_k|\alpha_{k-1}]=\E_{\alpha_{[k]}}[\Gamma'_k|x_k]=\E_{\alpha[k]}[\tilde G-c_2\nabla f|x_k]\stackrel{(e)}{\leq} \tau\mathbf{1}_{d\times1},\] 
with  $\Gamma'_k \equiv \tilde G(x_k,\xi_k,\xi_k^+,u_k,\delta)-c_2\nabla f (x) \equiv \tilde G (x_k,\alpha_k) - c_2\nabla f(x_k)$ and $\tau=c_2''\delta^2$  and $ \alpha_k \equiv (\xi_k,\xi_k^+,u_k,\delta), k\geq 1$. In the above vector inequality $(e)$ implies that $x_i\geq y_i$ for $X,Y\in\R^d$. Also, by \cref{martingle} we have
\[\E_{\alpha_{[k]}}[\lVert \tilde G \rVert^2]\leq \lVert\E_{\alpha_{[k]}}[ \tilde G]\rVert^2+\frac{C''}{\delta^2}.\]
The rest proof is similar as \cref{coro2}.
\end{proof}

\subsubsection{Proof of Theorem \ref{prop6}}
\label{pf:th6}
\begin{lemma}
\label{prop5}
Let $\nabla f_\delta(x)$ is the balanced estimator defined in \cref{eq-3} then under \cref{def:ass4} we have
\begin{equation}
    \lVert \nabla f_\delta(x)\rVert^2\leq 2\lVert \nabla f(x)\rVert^2 c_{11}^2+2 d^2\delta^4 c_2''^2
\end{equation}
\end{lemma}

\begin{proof}
From the definition of balanced estimator we have \\
\centerline{$\nabla f_\delta (x_k)
= {\displaystyle 
\E_u\bigg[\bigg(\frac{f(x_k+\delta_k u_k)-f(x_k-\delta_k u_k)}{2\delta_k}\bigg)\frac{(d+1)u_k}{1+\lVert u_k \rVert^2}}\bigg].$}\\
By Taylor series expansion we obtain
\begin{equation*}
    \begin{split}
   \nabla f_\delta (x_k)&= \E_u\bigg[\bigg(\frac{2\delta u_k^T\nabla f(x_k)+ \frac{\delta^3}{6}(\nabla^3f(\Bar{x_k}^+)+\nabla^3f(\Bar{x_k}^-))(u_k\otimes u_k \otimes u_k)}{2\delta}\bigg)\frac{(d+1)u_k}{1+\lVert u_k \rVert^2}\bigg]\\
   &= \E_u\Big[\frac{(d+1)u_ku_k^T\nabla f(x_k)}{1+\lVert u_k \rVert^2}\Big]+\E_u\Big[\frac{\delta^2(d+1)u_k}{12(1+\lVert u_k \rVert^2)}(\nabla^3f(\Bar{x_k}^+)+\nabla^3f(\Bar{x_k}^-))(u_k\otimes u_k \otimes u_k)\Big],\\
   &\leq \E_u\Big[\frac{(d+1)u_ku_k^T\nabla f(x_k)}{1+\lVert u_k \rVert^2}\Big]+c_2''\delta^2\mathbf{1}_d
        % f(x\pm\delta u)=f(x)\pm \delta u^T \nabla f(x)+\frac{\delta^2}{2}u^T\nabla^2f(x)u\pm\frac{\delta^3}{6}\nabla^3f(\Bar{x}^\pm)(u\otimes u \otimes u)
    \end{split}
\end{equation*}
The second term in the above inequality is obtain via same logic used in \cref{prop3}. Hence by taking norm of the both side and applying Jensen inequality we get 
\begin{equation*}
    \begin{split}
        \lVert \nabla f_\delta(x)\rVert &\leq \lVert \nabla f(x)\rVert \E\bigg[\frac{\lVert u \rVert^2(d+1)}{1+\lVert u \rVert^2}\bigg]+dc_2''\delta^2,\\
        &< \lVert \nabla f(x)\rVert \E\bigg[\lVert u \rVert^2(d+1)\bigg]+dc_2''\delta^2,\\
        & \leq \lVert \nabla f(x)\rVert c_{11}+dc_2''\delta^2,\\
        % & < \lVert \nabla f(x)\rVert c_{11}+\lVert \nabla^3 f(x) \rVert \delta^2 c_{11}\\
        \lVert \nabla f_\delta(x)\rVert^2 & \leq 2\lVert \nabla f(x)\rVert^2 c_{11}^2+2 d^2\delta^4 c_2''^2.\\
        % &\leq 2\lVert \nabla f(x)\rVert^2 c_{11}^2+2B_1 \delta^4c_{11}^2\\
        \end{split}
\end{equation*}
\end{proof}

\subsubsection*{Proof of \cref{prop6}}
\begin{proof}
Define $ \alpha_k \equiv (\xi_k,\xi_k^+,u_k), k\geq 1$ and $\alpha_{[N]}:=(\alpha_1,\alpha_2,...\alpha_N)$ as before in \cref{pf:Them2} \ Using Taylor series expansion over $f(x_k)$ we obtain for any $k=1,2,...,N$,
\begin{equation*}
    \begin{split}
        f(x_{k+1}) &\leq f(x_k)-\gamma_k\langle\nabla f(x_k),\tilde G(x_k,\alpha)\rangle+\frac{L}{2}\gamma_k^2 \left\lVert \tilde G(x_k,\alpha_k)\right\rVert^2\\
        &=f(x_k)-c_2\gamma_k\left\lVert\nabla f(x_k)\right\rVert^2- \gamma_k\langle \nabla f(x_k) ,\Gamma'_k\rangle
        +\frac{L}{2}\gamma_k^2 \left\lVert \tilde G(x_k,\alpha_k)\right\rVert^2.
    \end{split}
\end{equation*}
Here $\Gamma'_k \equiv \tilde G(x_k,\xi_k,\xi_k^+,u_k,\delta)-c_2\nabla f (x) \equiv \tilde G (x_k,\alpha_k) - c_2\nabla f(x_k)$.
Adding upto $N$-terms both side of these inequalities and applying $f^*\leq f(x_{N+1})$,
we obtain
\begin{align*}
        \sum_{k=1}^N c_2\gamma_k \left\lVert\nabla f(x_k)\right\rVert^2
        & \leq f(x_1)-f^*-\sum_{k=1}^N \gamma_k \langle\nabla f(x_k),\Gamma'_k\rangle +\frac{L}{2}\sum_{k=1}^N \gamma_k ^2\left\lVert \tilde G(x_k,\alpha_k)\right\rVert^2.
\end{align*}

Notice by \cref{prop3}
\[\E_{\alpha_{[k]}}[\Gamma'_k] =\E_{\alpha_{[k]}}[\Gamma'_k|\alpha_{k-1}]=\E_{\alpha_{[k]}}[\Gamma'_k|x_k]=\E_{\alpha[k]}[\tilde G-c_2\nabla f|x_k]\stackrel{(e)}{\leq} \tau\mathbf{1}_{d\times1},\] 
where $\Gamma'_k \equiv \tilde G(x_k,\xi_k,u_k)-c_2\nabla f (x) \equiv \tilde G (x_k,\alpha_k) - c_2\nabla f(x_k)$ and $\tau=c_2''\delta^2$. In the above vector inequality $(e)$ implies that $x_i\geq y_i$ for $X,Y\in\R^d$.

Now by \cref{def:ass10} and \cref{prop5} 
\begin{equation}
\begin{split}
    \E[\left\lVert G(x_k,\alpha_k) \right\rVert^2|\alpha_{[k-1]}] 
    &\leq \left[2c^2_{11}\E[\lVert \nabla F(x_k,\xi_k) \rVert^2|\alpha_{[k-1]}]+2d^2\delta^4c_2''^2\right]\\
     &\leq  \left[4c^2_{11}\E[\left\lVert \nabla f(x_k) \right\rVert^2|\alpha_{[k-1]}+\sigma^2]+2d^2\delta^4c_2''^2\right].
\end{split}
\end{equation}
Thus we have
\begin{flalign*}
       \sum_{k=1}^N c_2\gamma_k \E_{\alpha_{[N]}}\left\lVert\nabla f(x_k)\right\rVert^2
        & \leq D+B\tau\sum_{k=1}^N \gamma_k +\frac{L}{2}\sum_{k=1}^N \gamma_k^2 \left[4c^2_{11}\E[\left\lVert \nabla f(x_k) \right\rVert^2|\alpha_{[k-1]}+\sigma^2]+2d^2\delta^4c_2''^2\right].
\end{flalign*}

The above inequality uses the fact $-\lVert V \rVert_1\leq \sum_{k=1}^{d} v_k$ for a $d-$dimensional vector $V$ followed by $\lVert \nabla f(x_k)\rVert_1\leq \lVert \nabla f(x_k)\rVert\leq B$ in \cref{def:ass1}. Note that $D=f(x_1)-f^*$.
By rearranging the terms we have
\begin{flalign*}
        \sum_{k=1}^N  \left[c_2\gamma_k-2Lc^2_{11}\gamma_k^2\right]\E_{\alpha_{[N]}}\left\lVert\nabla f(x_k)\right\rVert^2
        &\leq D +B\tau\sum_{k=1}^N \gamma_k +L\left(4c_{11}\sigma^2+ 2d^2\delta^4c_2''^2\right)\sum_{k=1}^{N} \gamma_k^2.
\end{flalign*}
By the same argument as in \cref{eq100} and under the probability distribution 
\begin{equation}
    \p_R(k)=Prob(R=k)=\frac{\left[c_2\gamma_k-2Lc^2_{11}\gamma_k^2\right]}{\sum_{k=1}^{N}\left[c_2 \gamma_k-2Lc^2_{11}\gamma_k^2\right]},
\end{equation}
the following is obtained
\begin{flalign*}
         \E_{\alpha_{[N]}}\left\lVert\nabla f(x_R)\right\rVert^2
        &\leq\frac{1}{\sum_{k=1}^{N}\left[c_2\gamma_k-Lc^2_{11}\gamma_k^2\right]}\bigg[D +B\tau\sum_{k=1}^N \gamma_k +L\left(4c_{11}\sigma^2+ 2d^2\delta^4c_2''^2\right)\sum_{k=1}^{N} \gamma_k^2.\bigg].
\end{flalign*}
Note that, the condition of $\gamma_k$, see \cref{eq101}, is given by
\[
    \gamma_k =\min\left\{\frac{c_2}{4c^2_{11}L},\frac{1}{N^{1/2}}\right\}, \quad k=1,2,...N,
\]Thus one can have
\begin{equation*}
    \sum_{k=1}^N\left[c_2\gamma_k -2Lc^2_{11}\gamma_k^2\right]=Nc_2\gamma_1\left[1 -\frac{2Lc^2_{11}}{c_2}\gamma_1\right]\geq \frac{Nc_2\gamma_1}{2}.
\end{equation*}
From the above inequality, we can write
\begin{flalign*}
\E\left[\left\lVert\nabla f(x_R)\right\rVert^2\right]
        & \leq \frac{2D} {Nc_2\gamma_1}+\frac{2B\tau}{c_2}+\frac{L}{c_2}\left(4c_{11}\sigma^2+ 2d^2\delta^4c_2''^2\right)\gamma_1 ,\\
         &\leq \frac{2D}{Nc_2}\max\left\{\frac{4c^2_{11}L}{c_2}, N^{1/2}\right\}+\frac{2B\tau}{c_2}+\frac{L}{c_2N^{1/2}}\left(4c_{11}\sigma^2+ 2d^2\delta^4c_2''^2\right),\\
         &\stackrel{(h)}{\leq} \Big(\frac{2DL}{Nc_2^2}+ \frac{2D}{N^{1/2}}\Big)+\frac{2Bc_2''}{c_2N}+\frac{L}{2c_2N^{1/2}}\left(4c_{11}\sigma^2+\frac{2d^2c_2''}{N^2}\right).\\
\end{flalign*}
Note that $(h)$ uses the condition of $\delta=\frac{1}{N^{1/2}}$ by putting $\tau=c_2''\delta^2$. Thus, we get a rate of convergence of $\mathcal{O}(N^{-1/2}).$
\end{proof}

\section{Experiments}
\label{experiments}
In this section, we compare the performance of GSF, SPSA with symmetric Bernoulli($\pm1$) valued perturbations and RDSA  
with uniform ($-5,5$) perturbations, respectively,
with the TCSF (\cref{alg:alg1}) and TCSF with balanced estimator\cref{balanced} (B-TCSF). We consider both non-convex and convex objective functions with additive noise. 
We consider the following choices for the function $F(x,\xi)$ with $d=4$:

\begin{table}[H]
  \caption{The different models for the observations $F(\cdot,\cdot)$ considered in our experiments.}
  \label{sample-table}
  \centering
  \begin{tabular}{llll}
    % \toprule
    % \multicolumn{2}{c}{Part}                   \\
    \cmidrule(r){1-4}
    Name     & Functional form     & Optimal point $x^*$ & Min-Value\\
    \midrule
    Rastrigin & $10d+\sum_{i=1}^d(x_i^2-10\cos(2\pi x_i))+\xi_x$  & $(0,\ldots,0)^T$  & 0   \\
    Rosenbrock     & $\sum_{i=1}^{d-1}(100(x_{i+1}-x_i^2)^2+(1-x_i)^2)+\xi_x$, & $(1,\ldots,1)^T$  & 0    \\
    Quadratic  & $\frac{1}{2}x^TAx-b^Tx+\xi_x$       & $(-135.1,\ldots,-5.6)^T$ & $-17.528$ \\
    \bottomrule
  \end{tabular}
\end{table}

The setting for the case of quadratic function considered in \cref{sample-table} is \[F_3(x,\xi_x)=\frac{1}{2}x^TAx-b^Tx+\xi_x,\] with, $x^*=[-135.1150,-4.5224,130.1168,-5.6879]^T$ where
\[ A= \begin{bmatrix}
	2.3346 & 1.1384 & 2.5606 & 14507 \\
	1.1384 & 0.7860 & 1.2743 & 0.9531 \\
	2.5606 & 1.2743 & 2.8147 & 1.6487 \\
	1.4507 & 0.9531 & 1.6487 & 1.8123
\end{bmatrix}\] and 
\[b=[0.4218,
0.9157,0.7922,0.9595]^T.\]
We consider three settings for the noise $\eta$.
In the first setting, referred to as Type-1, we let $\xi_x=[x^T,1]\eta$,  where $\eta$ is a multivariate Gaussian with zero mean, and covariance matrix $\sigma^2\mathbb{I}_{d+1}$ with $\sigma=5$.
In the second setting, referred to as  Type-2, we have $\xi_x$ as a Gaussian random variable with mean $0$ and variance $\ln\lVert x \rVert_2$.
Finally, in the last setting, referred to as Type-3,  we have $\xi_x$ as a Gaussian random variable with mean $0$ and variance $\frac{1}{1+\ln\lVert x \rVert_2}$.
Note in particular that Rosenbrock is a badly-scaled function, while Rastrigin is a multi-modal function.

For the truncated Cauchy perturbations in TCSF and B-TCSF,
we generated samples from the multivariate t-distribution with one degree of freedom and then projected the same to the unit sphere.
In our experiments, we set $\epsilon=0.0001$, i.e., we stop the algorithm when $\lVert G\rVert \leq \epsilon$. For our initial experiments, we used the stepsize and smoothing parameter 
	as follows: $\gamma_k=\frac{1}{k^{0.6}}$ and $\delta_k=\frac{1}{k^{0.09}}$. However,  we used constant step-sizes $ 0.0001$ and smoothing parameters $ 0.001$, respectively, which work well on both algorithms. We run each algorithm 100 times with $N=1000$, $3000$, $10000$, respectively, for the Rastrigin, quadratic and Rosenbrock functions, and the averages of the optimal functional values are reported in \cref{table2} while the standard error estimates for the various algorithms from the different simulation runs are given in \cref{table4}. we considered the set  $[0,10]^4$ for Rastrigin and Rosenbrock functions, and $[0,150]^4$ for a quadratic objective function. We report the average number of iterations needed to reach an $\epsilon$-stationary point in \cref{table1} 

\begin{table}[H]
   \caption{Average functional values for five SG algorithms with diminishing step size and smoothing parameter}
   \label{table2}
  \centering
  \begin{tabular}{p{1cm}>{\bfseries}cccccc}
     \cmidrule(lr){1-7}
    & &  \thead{GSF} & \thead{TCSF} &\thead{B-TCSF} & \thead{SPSA} & \thead{RDSA}  \\
    \midrule
    & Rastrigin &   $0.00019$ \ \ &$1.03e-05$ \ \ &$0.0$ \ \ & $0.0092$ \ \ & $0.0094$  \\
     & Rosenbrock &   $0.002$ \ \ & $0.00066$ \ \ & $0.00062$ \ \ & $0.0010$ \ \ & $0.0017$\\
     \raisebox{10pt}[0pt][0pt]{\rotatebox[origin=c]{90}{\parbox{2.5cm}{\centering Type 1}}}  & Quadratic &   $-17.471$ \ \ & $-17.513$ \ \ & $-17.5286$ \ \  & $-17.4719$ \ \ & $-17.4731$ \\
    \bottomrule
    & Rastrigin &   $0.0097$ \ \ &$0.0009$ \ \ & $1.17e-05$ \ \ & $0.0096$ \ \ & $0.01$ \\
     & Rosenbrock &   $0.003$ \ \ & $0.0018$ \ \ & $0.00091$ \ \ & $8e-05$ \ \ & $0.0017$\\
     \raisebox{10pt}[0pt][0pt]{\rotatebox[origin=c]{90}{\parbox{2.5cm}{\centering Type 2}}}  & Quadratic &   $-17.5286$ \ \ & $-17.528$ \ \ & $-17.5248$ \ \ & $-17.52819$ \ \ & $-17.474$ \\
    \bottomrule
    & Rastrigin &   $0.0002$ \ \ &$1.11e-5$ \ \ & $0.0$ \ \ & $0.0$ \ \ & $0.009$ \\
     & Rosenbrock &   $0.017$ \ \ & $0.00072$ \ \ & $0.00026$ \ \ & $0.0021$ \ \ & $0.0017$\\
     \raisebox{10pt}[0pt][0pt]{\rotatebox[origin=c]{90}{\parbox{2.5cm}{\centering Type 3}}}  & Quadratic &   $-17.5253$ \ \ & $-17.5248$ \ \ & $-17.5279$ \ \ & $-17.492$ \ \ & $-17.488$ \\
    \bottomrule
  \end{tabular}
\end{table}

\begin{table}[H]
   \caption{Standard error for \cref{table2}}
   \label{table4}
  \centering
  \begin{tabular}{p{1cm}>{\bfseries}cccccc}
     \cmidrule(lr){1-7}
    & &  \thead{RSGF} & \thead{TCSF} &\thead{B-TCSF} & \thead{SPSA} & \thead{RDSA}  \\
    \midrule
   & Rastrigin    & $3.44\times 10^{-7}$ \ \ &$2.34\times 10^{-8}$ \ \ & $0$ \ \ & $9.41\times 10^{-6}$  \ \ & $4.84\times 10^{-6}$\\
    & Rosenbrock  & $4.96\times 10^{-5}$ \ \ &$3.75\times 10^{-7}$ \ \ & $1.24\times 10^{-7}$ \ \ & $6.34\times 10^{-6}$  \ \ & $1.51\times 10^{-6}$\\
     \raisebox{10pt}[0pt][0pt]{\rotatebox[origin=c]{90}{\parbox{2.5cm}{\centering Type 1}}}  & Quadratic     & $5.61\times 10^{-6}$ \ \ &$1.58\times 10^{-7}$ \ \ & $1.28\times 10^{-7}$ \ \ & $64.25\times 10^{-5}$  \ \ & $5.28\times 10^{-6}$\\
    \bottomrule
    & Rastrigin     & $1.26\times 10^{-5}$ \ \ &$1.62\times 10^{-7}$ \ \ & $2.96\times 10^{-7}$ \ \ & $4.29\times 10^{-5}$  \ \ & $5.61\times 10^{-6}$\\
    & Rosenbrock  & $1.59\times 10^{-6}$ \ \ &$4.57\times 10^{-7}$ \ \ & $4.29\times 10^{-7}$ \ \ & $4.27\times 10^{-6}$  \ \ & $1.59\times 10^{-6}$\\
     \raisebox{10pt}[0pt][0pt]{\rotatebox[origin=c]{90}{\parbox{2.5cm}{\centering Type 2}}}  & Quadratic     & $4.36\times 10^{-7}$ \ \ &$1.54\times 10^{-7}$ \ \ & $2.03\times 10^{-6}$ \ \ & $6.29\times 10^{-6}$  \ \ & $5.69\times 10^{-6}$\\
    \bottomrule
     & Rastrigin     & $4.29\times 10^{-6}$ \ \ &$1.80\times 10^{-7}$ \ \ & $0$ \ \ & $0$  \ \ & $2.49\times 10^{-5}$\\
    & Rosenbrock  & $2.41\times 10^{-6}$ \ \ &$1.94\times 10^{-5}$ \ \ & $1.62\times 10^{-7}$ \ \ & $2.36\times 10^{-5}$  \ \ & $6.16\times 10^{-5}$\\
     \raisebox{10pt}[0pt][0pt]{\rotatebox[origin=c]{90}{\parbox{2.5cm}{\centering Type 3}}}  & Quadratic     & $3.26\times 10^{-6}$ \ \ &$2.63\times 10^{-6}$ \ \ & $1.30\times 10^{-7}$ \ \ & $6.11\times 10^{-5}$  \ \ & $8.83\times 10^{-5}$\\
    \bottomrule
  \end{tabular}
\end{table}

There is a significant difference in optimal functional values obtained from TCSF, B-TCSF as compared to the other algorithms. One can notice from the \cref{table2} that $|f(x^*)-f(\Bar{x}_{sol})|$ $>$ $|f(x^*)-f(\hat{x}_{{sol}})|$, where $\Bar{x}_{sol}$ indicates the final output from GSF, SPSA, RDSA and $\hat{x}_{{sol}}$ denotes the final output from TCSF, B-TCSF. Only in the case of Rosenbrock and Rastrigin  with Type 2 and Type 3 error, SPSA works slightly better than TCSF though not so when compared with B-TCSF. However in case of the quadratic function we have noticed that TCSF beats other algorithms under Type-1 and Type-3 errors as well. 
Between TCSF and B-TCSF, B-TCSF is seen to perform slightly better on the whole. This is to be expected since B-TCSF has a lower bias because of the direct cancellation of even-order bias terms starting from the second order. 

\begin{table}[H]
   \caption{Average number of iterations to converge to the optimal point(constant step size)}
   \label{table1}
  \centering
  \begin{tabular}{p{1cm}>{\bfseries}cccccc}
     \cmidrule(lr){1-7}
    & &  \thead{RSGF} & \thead{TCSF} &\thead{B-TCSF} & \thead{SPSA} & \thead{RDSA}  \\
    \midrule
    & Rastrigin &   $837.6$ \ \ &$331.5$ \ \ &$149.6$ \ \ & $1268.39$ \ \ & $1281.39$  \\
     & Rosenbrock &   $5604.08$ \ \ & $3995.2$ \ \ & $3138.3$ \ \ & $6968.41$ \ \ & $8572.53$\\
     \raisebox{10pt}[0pt][0pt]{\rotatebox[origin=c]{90}{\parbox{2.5cm}{\centering Type 1}}}  & Quadratic &   $8727.82$ \ \ & $3234.86$ \ \ & $2994.15$ \ \  & $5952.7$ \ \ & $6000.3$ \\
    \bottomrule
    & Rastrigin &   $785.6$ \ \ &$638.9$ \ \ & $151.4$ \ \ & $1269.32$ \ \ & $1105.33$ \\
     & Rosenbrock &   $5924.5$ \ \ & $4834.2$ \ \ & $2231.2$ \ \ & $1782.1$ \ \ & $9419.4$\\
     \raisebox{10pt}[0pt][0pt]{\rotatebox[origin=c]{90}{\parbox{2.5cm}{\centering Type 2}}}  & Quadratic &   $3305.44$ \ \ & $3050.96$ \ \ & $2863.11$ \ \ & $3154.49$ \ \ & $5743.9$ \\
    \bottomrule
    & Rastrigin &   $896.8$ \ \ &$398.6$ \ \ & $128.7$ \ \ & $137.6$ \ \ & $1352.4$ \\
     & Rosenbrock &   $5450.51$ \ \ & $3376.6$ \ \ & $2735.2$ \ \ & $7946.4$ \ \ & $8847.3$\\
     \raisebox{10pt}[0pt][0pt]{\rotatebox[origin=c]{90}{\parbox{2.5cm}{\centering Type 3}}}  & Quadratic &   $3518.9$ \ \ & $3319.68$ \ \ & $2726.65$ \ \ & $4154.49$ \ \ & $6152.2$ \\
    \bottomrule
  \end{tabular}
\end{table}

\begin{table}[H]
   \caption{Standard error for \cref{table1}}
   \label{table5}
 \centering
  \begin{tabular}{p{1cm}>{\bfseries}cccccc}
     \cmidrule(lr){1-7}
    & &  \thead{RSGF} & \thead{TCSF} &\thead{B-TCSF} & \thead{SPSA} & \thead{RDSA}  \\
    \midrule
    & Rastrigin &   $4.79$ \ \ &$0.325$ \ \ &$0.1255$ \ \ & $0.395$ \ \ & $0.445$  \\
     & Rosenbrock &   $90.4$ \ \ & $38.35$ \ \ & $20.3$ \ \ & $108.2$ \ \ & $115.3$\\
     \raisebox{10pt}[0pt][0pt]{\rotatebox[origin=c]{90}{\parbox{2.5cm}{\centering Type 1}}}  & Quadratic &   $31.85$ \ \ & $24.49$ \ \ & $17.85$ \ \  & $1.07$ \ \ & $1.0407$ \\
    \bottomrule
    & Rastrigin &   $5.66$ \ \ &$5.412$ \ \ & $0.14$ \ \ & $0.445$ \ \ & $0.415$ \\
     & Rosenbrock &   $108.08$ \ \ & $45.36$ \ \ & $21.71$ \ \ & $13.99$ \ \ & $99.35$\\
     \raisebox{10pt}[0pt][0pt]{\rotatebox[origin=c]{90}{\parbox{2.5cm}{\centering Type 2}}}  & Quadratic &   $20.78$ \ \ & $20.85`$ \ \ & $15.18$ \ \ & $0.895$ \ \ & $1.13$ \\
    \bottomrule
    & Rastrigin &   $6.114$ \ \ &$0.319$ \ \ & $0.132$ \ \ & $0.128$ \ \ & $0.431$ \\
     & Rosenbrock &   $102.21$ \ \ & $22.93$ \ \ & $19.02$ \ \ & $103.36$ \ \ & $121.3$\\
     \raisebox{10pt}[0pt][0pt]{\rotatebox[origin=c]{90}{\parbox{2.5cm}{\centering Type 3}}}  & Quadratic &   $22.17$ \ \ & $24.18$ \ \ & $19.15$ \ \ & $1.04$ \ \ & $0.99$ \\
    \bottomrule
  \end{tabular}
\end{table}

In \cref{table1} above we have described the number of iterations needed to reach an $\epsilon-$stationary point while \cref{table5} describes the SE for corresponding iteration . There is a significant difference in the number of iterations required for converging to the optimal point in each of the cases. For example, GSF, RDSA and SPSA each take more than 2000 iterations with respect to TCSF, B-TCSF. However SPSA performs well as compared to RDSA, RSGF and in some experiments it beats TCSF (e.g., Rosenbrock with type-2 error). 

\begin{table}[H]
   \caption{Average functional values for five SG algorithm with constant step size and smoothing parameter}
   \label{table3}
  \centering
  \begin{tabular}{p{1cm}>{\bfseries}cccccc}
     \cmidrule(lr){1-7}
    & &  \thead{RSGF} & \thead{TCSF} &\thead{B-TCSF} & \thead{SPSA} & \thead{RDSA}  \\
    \midrule
   & Rastrigin    & $1.722$ \ \ &$0.091$ \ \ & $0.008$ \ \ & $0.987$  \ \ & $2.024$\\
    & Rosenbrock & $1.954$ \ \ & $0.013$ \ \ & $0.0006$ \ \ & $0.587$ \ \ & $1.743$\\
     \raisebox{10pt}[0pt][0pt]{\rotatebox[origin=c]{90}{\parbox{2.5cm}{\centering Type 1}}}  & Quadratic    & $-15.894$ \ \ & $-17.685$ \ \ & $-17.587$ \ \ & $-17.096$ \ \ & $-14.843$\\ 
    \bottomrule
    & Rastrigin    & $1.985$ \ \ &$0.065$ \ \ & $0.008$ \ \ & $0.848$ \ \ & $2.542$\\
    & Rosenbrock & $1.926$ \ \ & $0.0018$ \ \ & $0.0089$ \ \ & $0.683$ \ \ & $2.193$ \\
     \raisebox{10pt}[0pt][0pt]{\rotatebox[origin=c]{90}{\parbox{2.5cm}{\centering Type 2}}}  & Quadratic    & $-14.448$ \ \ & $-17.939$ \ \ & $-17.373$ \ \ & $-16.597$ \ \ & $-13.2493$\\
    \bottomrule
     & Rastrigin    & $1.875$ \ \ &$0.01$ \ \ & $0.00623$ \ \ & $0.879$ \ \ & $1.893$\\
    & Rosenbrock & $2.097$ \ \ & $0.091$ \ \ & $0.0034$ \ \ & $0.698$ \ \ & $2.314$\\
     \raisebox{10pt}[0pt][0pt]{\rotatebox[origin=c]{90}{\parbox{2.5cm}{\centering Type 3}}}  & Quadratic    & $-14.869$ \ \ & $-17.962$ \ \ & $-17.459$ \ \ & $-16.183$ \ \ & $-13.743$\\
    \bottomrule
  \end{tabular}
\end{table}
We can observe from \cref{table3} that for constant step size and smoothing parameter $f(\Bar{x}_{TCSF})$,$f(\Bar{x}_{B-TCSF})$ provides almost optimal functional value as compared to other algorithms. However, B-TCSF is more efficient than TCSF in this context.
% We can observe from \cref{table3} that for diminishing step size and smoothing parameter $f(\Bar{x}_{GSF})$,$f(\Bar{x}_{RDSA})$ is much higher than the optimal functional value however in case of TCSF,B-TCSA,SPSA, the optimal function value $f(\Bar{x}_{TCSF})$ is not too far from optimal functional value. However among TCSF,B-TCSF and SPSA clearly B-TCSF performs better than other TCSF and SPSA on average.\\
Thus we can conclude that there is an improvement in empirical performance when TCSF and B-TCSF are used over other SG algorithms. 
\section{Conclusions and future work}
We proposed and analyzed a gradient estimation scheme, based on truncated Cauchy random perturbations, for solving a non-convex smooth optimization problem. We showed that our algorithm avoids traps and converges asymptotically to a local minimum. Our algorithm performs better than two popular gradient estimation schemes in the literature, namely SPSA and GSF, in terms of the asymptotic convergence rate. We also provided non-asymptotic rate for our algorithm that is the same as the asymptotic rate and better when common random noise is used in the simulations. Our algorithm also performs better than GSF, SPSA, RDSA empirically. 
Exploring the performance of the Newton method using Hessian estimation under the truncated Cauchy perturbations would be an interesting direction for future work. 
\appendix

\bibliography{refs}

\end{document}